\documentclass[11pt,reqno]{amsart} 
\usepackage{mathrsfs}
\usepackage{amssymb,amsmath,amsthm,color}
\usepackage{graphicx}
\usepackage[colorlinks=true,linkcolor=blue,citecolor=green]{hyperref}
\usepackage{url} 
\usepackage{verbatim}
\usepackage{setspace}
\usepackage{mathtools}
\usepackage{multirow}
\usepackage{array}
\usepackage{stackengine}
\usepackage{tikz}
\usepackage{cite}
\usepackage[margin=1in]{geometry}
\usepackage{comment}
\numberwithin{equation}{section}
\newtheorem{theorem}{Theorem}[section]
\newtheorem{Definition}[theorem]{Definition}
\newtheorem{Remark}[theorem]{Remark}
\newtheorem{proposition}[theorem]{Proposition}
\newtheorem{Lemma}[theorem]{Lemma}

\newcommand{\R}{\mathbb R}
\newcommand{\rd}{\mathbb{R}^d}
\newcommand{\n}{\Vert}
\newcommand{\C}{\mathbb C}
\newcommand{\F}{\mathcal{F}}

\newcommand{\w}{\widetilde}
\newcommand{\h}{\mathcal{H}}

\newcolumntype{P}[1]{>{\centering\arraybackslash}p{#1}}
\newcolumntype{M}[1]{>{\centering\arraybackslash}m{#1}}
\usepackage[pagewise]{lineno}

\def\({\left(}
\def\){\right)}
\def\<{\left\langle}
\def\>{\right\rangle}

\begin{document}
\baselineskip16pt
\title[Hartree-Fock equations in  $L^{p}$ and $\widehat{L}^p$ ]{The Hartree and  Hartree-Fock equations in  Lebesgue $L^p$ and Fourier-Lebesgue $\widehat{L}^p$ spaces}
\author[D. G. Bhimani]{Divyang G. Bhimani}
\address{Department of Mathematics, Indian Institute of Science Education and Research, Dr. Homi Bhabha Road, Pune 411008, India}
\email{divyang.bhimani@iiserpune.ac.in}
\author[S. Haque]{Saikatul Haque*}
\address{Harish-Chandra Research Institute, Chhatnag Road, Jhunsi, Prayagraj (Allahabad) 211019,
 India}
\email{saikatulhaque@hri.res.in}
\subjclass[2010]{35Q40, 35Q55, 42B35 (primary), 35A01 (secondary)}
\keywords{reduced Hartree-Fock and Hartree-Fock equations, local and global well-posedness}

\begin{abstract}  
We establish  some local and global well-posedness  for  Hartree-Fock   equations   of $N$ particles (HFP) with Cauchy data in  Lebesgue   spaces $L^p \cap L^2 $  for  $1\leq p \leq \infty$.   Similar results  are  proven for fractional HFP in    Fourier-Lebesgue  spaces $ \widehat{L}^p \cap L^2 \ (1\leq p \leq \infty).$   On the other  hand, we show that the Cauchy problem for HFP is mildly ill-posed  if  we simply work in $\widehat{L}^p \ (2<p\leq \infty).$  Analogue results hold for  reduced HFP.
In the process, we prove the boundedeness of  various   trilinear  estimates for Hartree type non linearity  in these  spaces which may be of independent interest.
As a consequence, we get natural   $L^p$  and $\widehat{L}^p$ extension of  classical  well-posedness theories  of Hartree and Hartree-Fock equations  with Cauchy data in just $L^2-$based Sobolev spaces.

\end{abstract}
\maketitle

\section{Introduction}

\subsection{Motivation and physical context}
The Hartree-Fock equation (HFE),   defined in  \eqref{hf},  is a key effective equation of quantum physics. It plays a role similar to that of the Boltzmann equation in classical physics. The HFE describes large systems of identical fermions by taking into account the self-interactions of charged fermions as well as an exchange term resulting from Pauli's principle. A semirelativistic version of the HFE was developed in  \cite{frohlich2007dynamical}   for modeling white dwarfs.  The HFE model \cite{lewin2018existence} leads to the Kohn-Sham equation underlying the density functional theory which is exceptionally effective in computations in quantum chemistry and in particular, of the electronic structure of matter.  The HFE is used for several applications in many-particle physics \cite{lipparini2008modern}. For  detailed  background and recent developments  on HFE and beyond, we  refer to the excellent survey  \cite{chenn2019effective} and the references therein.

In   \cite{ laskin2002fractional}   fractional Laplacians have been applied to model physical phenomena.  This was formulated by Laskin \cite{laskin2002fractional} as a result of extending the Feynman path integral from the Brownian-like to L\'evy-like quantum mechanical paths.
Specifically, when  $\alpha=1,$  the fractional Hartree equation,  defined  in \eqref{rhf},  can be used to describe the dynamics of pseudo-relativistic boson stars in the mean-field limit, and when $\alpha=2$ the L\'evy motion becomes Brownian motion. The Hartree equation also arises in the nonlinear optics of nonlocal, nonlinear optical media  \cite{peccianti2012nematicons}.

\subsection{Hartree-Fock equations}\label{shf}
The Hartree-Fock equations of $N$ particles is given by
\begin{equation}\label{hf}
\begin{cases} i\partial_t \psi_k - (-\Delta)^{\alpha/2}\psi_k+\kappa\sum_{l=1}^{N} \left(\frac{e^{-a|x|}}{|x|^{\gamma}} \ast |\psi_l|^{2} \right)\psi_{k} - \kappa\sum_{l=1}^{N} \psi_{l} \left(\frac{e^{-a|x|}}{|x|^{\gamma}} \ast  (\overline{\psi_l} \psi_k)\right)=0\\
\psi_{k|t=0}=\psi_{0,k}
\end{cases}
\end{equation}
where $a\geq0$, $t\in \R, \psi_k:\R^d\times \R \to \C,$ $k=1,2,..., N$, $0<\gamma<d,$  $ \kappa$ is constant,  and $\ast$ denotes the convolution in $\mathbb R^d.$ 
The fractional Laplacian is defined as 
\begin{eqnarray*}
\mathcal{F}[(-\Delta)^{\alpha/2}u] (\xi) = (2\pi)^\alpha |\xi|^{\alpha} \mathcal{F}u (\xi),\quad 0<\alpha < \infty
\end{eqnarray*}
where $\mathcal{F}$ denotes the Fourier transform. 
The Hartree factor 
\begin{eqnarray*}
H_\psi= \kappa\sum_{l=1}^{N} \left(\frac{e^{-a|x|}}{|x|^{\gamma}} \ast |\psi_l|^2 \right)
\end{eqnarray*}
describes the self-interaction between charged particles as a repulsive force if  $\kappa>0$, and an attractive force   if $\kappa<0.$  In  $H_\psi$ the cases $a=\gamma=1$ and $a=0, \gamma=1$  corresponds to, well-known, \textbf{Yukawa}  and \textbf{Coulomb potentials} respectively. The last term on the left side of  \eqref{hf} is the so-called ``exchange term (Fock term)"
\[F_\psi(\psi_k)=\kappa\sum_{l=1}^{N} \psi_{l} \left(\frac{e^{-a|x|}}{|x|^{\gamma}} \ast (\overline{\psi_l} \psi_k) \right)\]
 which is a consequence of the Pauli principle and thus applies to fermions.  In the mean-field limit $(N\to \infty)$, this term is negligible compared to the Hartree factor.
 In this case,~\eqref{hf} is replaced by the  $N$ coupled equations, the so-called  \textbf{reduced Hartree-Fock equations}:
\begin{equation}\label{rhf}
i\partial_t \psi_k - (- \Delta)^{\alpha/2} \psi_k+\kappa\sum_{l=1}^{N} \left(\frac{e^{-a|x|}}{|x|^{\gamma}} \ast |\psi_l|^2 \right)\psi_{k}=0,\quad
\psi_{k|t=0}=\psi_{0,k}.
\end{equation}
In particular, when $a=0, N=1,$ and $\alpha=2,$ \eqref{rhf} is the classical  \textbf{Hartree equation}. 
We denote by $(\#)$ either  \eqref{hf} with $N\geq 2$  or \eqref{rhf} with $N\geq1$.

  Fr\"ohlich-Lenzmann  \cite[Theorem 2.1]{frohlich2007dynamical} proved  that $(\#)$ with Coulomb type  self-interaction  is  locally well-posed in  $H^{s}(\R^3) \ (s\geq 1/2).$ Moreover, they  \cite[Theorem 2.2]{frohlich2007dynamical}  proved global  existence   for  sufficiently small initial data.    Carles-Lucha-Moulay \cite[Section IV]{carles2015higher}    studied  global well-posedness  of \eqref{hf}   for Coulomb type  self-interaction and with  an external potential,  and obtained  some $H^{s}(\R^3)$ regularity.
Lenzmann \cite[Theorems 1, 2 and  3]{lenzmann2007well} proved some local and global  well-posedness  for Hartree equation with Yukawa type self-interaction in  $H^{s}(\R^3)$ with $s\geq 1/2.$  See also \cite{cho2013cauchy, cho2017orbital,  tarulli2019decay}.

Thus most authors have studied well-posedness  for  the Cauchy problem of  $(\#)$ in  $L^2$-based Sobolev spaces. This is of course very crucial from the physical point of view. Also from the mathematical point of view, 
the major reason behind this is the fact that the free Schr\"odinger  propagator  $U(t):=e^{it\Delta}:L^{p}(\R^d) \to L^p({\R^d})$   if and only if $p=2.$ This raises a natural question: Can we expect well-posedness theory for $(\#)$  in function spaces$-$which are not just based on $L^2$-integrability?  The  fantastic progress has been made  for this in the last decade. In fact,  Zhou  \cite{zhou2010cauchy} proved some well-posedness for nonlinear Schr\"odiger equation (NLS)  in  some $L^p$-Sobolev spaces for $p<2.$
Then  local  well-posedeness  for  Hartree equation was  studied in \cite{hoshino2019trilinear, hyakuna2018multilinear, hyakuna2018global,  hyakuna2019global} with Cauchy data in $L^p(\rd)\cap L^2(\rd)$ for $1\leq p \leq 2$ (cf.   \cite{wang2007global,  ruzhansky2012modulation, bhimani2016functions}).
 The method of proofs employed  in these works heavily rely on the Zhou spaces \cite{zhou2010cauchy} functional frame work$-$which is interesting but quite technical in itself.  And  there is a subtle flaw in the proof in \cite{hyakuna2018global}, see Remark \ref{r2} \eqref{HF} for detail.  Besides, this  raises  question: can we extend  these works in $L^p(\rd)\cap L^2(\rd)$ for $2<p \leq \infty$? This has been widely open question since the pioneering work of Zhou for NLS in  \cite{zhou2010cauchy}.
The novelty of this paper  is that we could establish local well-posednenss for $(\#)$ in $L^p(\rd)\cap L^2(\rd)$  without using Zhou spaces functional frame work.   This new approach simplifies the method of proofs and   enables us to cover wider range $1\leq  p \le \infty$.  Theorem \ref{lw} thus  completes the study of local well-posedness for Hartree equation in Lebesgue spaces $L^p(\rd)\cap L^2 (\rd)$ for all $1\leq p \leq \infty$.    Specifically,  we have following theorem.
\begin{theorem}[Local well-posedness in $L^p\cap L^2$]\label{lw} 
Let $\gamma$ satisfy one of the followings
\begin{itemize}
\item $0<\gamma<   \min\{1,2d(\frac{1}{p}-\frac{1}{2})\}$\quad for\quad $1\leq p\leq\frac{4}{3}$
\item $0<\gamma< \min\{1, \frac{d}{2}\}$\quad for\quad $\frac{4}{3}\leq p\leq\infty$.
\end{itemize}
Assume  that  $\psi_0= (\psi_{0,1}, ..., \psi_{0,N})\in  \left( L^p(\rd)\cap L^2(\rd)\right)^N$ and $\alpha=2.$ Then there exists $T>0$ and a unique local solution $ (\psi_1,..., \psi_N)$ to  $(\#)$ such that
$$\left (U(-t)\psi_1(t),..., U(-t)\psi_N(t)\right) \in \big( C([0,T],L^p(\rd)\cap L^2(\rd)) \big)^N.$$
\end{theorem}
 
 The linear counterpart problem of  $(\#)$ (free Schr\"odinger equation)  is ill-posed in $L^p(\R^d)$ for $p\neq 2.$ But Theorem \ref{lw} reveals that after a linear transformation using the semigroup $U(-t)$ generated by the linear problem,  $(\#)$ is locally well-posed in $L^p(\rd)\cap L^2(\rd).$    We will give two different proofs of Theorem \ref{lw}. In the first proof  Zhou spaces  will play no role, this  contrasts with the proof given in \cite{hyakuna2018global} for the  Hartree equation. In the second proof we make use of Zhou spaces to get  the local existence. (Our approach differs from \cite{hyakuna2018global}, see Remark \ref{SR2} for detail.)  In this case,  local solutions enjoy some Zhou spaces regularity. 
 
\begin{Remark} \label{r2}
\begin{enumerate}
\item The proof of Theorem \ref{lw} relies on a  factorization formula for $U(t)$ stated in 
Lemma \ref{2}, trilinear estimates (Subsection \ref{TE}),  and Strichartz estimates.  For the detailed proof strategy, see  Remark \ref{mip}.

 \item\label{HF} Local well-posedness established  for Hartree equation in \cite[Theorem 1.3]{hyakuna2018global} in $L^p\cap L^2$ for $1<p<2$ used Zhou spaces approach together  with some clever ideas. However, there is a flaw in the proof, see Remark \ref{SR1} for detail.  We overcome this  issue by not decomposing Duhamel type operator $\mathcal{D}_{a, \gamma}$, see \eqref{dhtp}. See the proof of Lemma \ref{8} and Remark \ref{SR2}.

\item We do not know factorization formula for fractional Schr\"odinger propagator $e^{-it (-\Delta)^{\alpha/2}}$  $(\alpha \neq 2)$, and so the analogue of  Theorem \ref{lw} for $\alpha\neq 2$, remains an open question.

\end{enumerate}
\end{Remark}

The local solutions of Theorem \ref{lw} can be extended globally, under an additional assumption on $\gamma$. Specifically, we have the following theorem:

\begin{theorem}[Global well-posedness in  $L^p\cap L^2$]\label{global1}
Let $\alpha=2$ and 
$0<\gamma<\min\{1, \frac{d}{2}\}$.
Then the local solution to $(\#)$ given by Theorem \ref{lw}  extends to a global one such that 
$$\left (U(-t)\psi_1(t),..., U(-t)\psi_N(t)\right) \in \big( C(\R,L^p(\rd)\cap L^2(\rd)) \big)^N.$$
Moreover, it follows that  $ (\psi_1(t),..., \psi_N(t)) \in \left(C(\R, L^2(\rd)) \right)^N$ and that if $1\leq p\leq2$ then the global solution  enjoys the  following smoothing  effect in terms  of special integrability:
\[ (\psi_1(t),..., \psi_N(t))\mid_{(\R \setminus \{0\} \times \R^d)} \in \big(C(\R \setminus \{ 0\}, L^{p'}(\rd)) \big)^N.\]
\end{theorem}
We  note that formally the solution of (\#) satisfies  (see e.g, \cite{carles2015higher})
the conservation of mass:
\[\|\psi_k(t)\|_{L^2}=\|\psi_{0,k}\|_{L^2}  \   \   ( t\in \R, k=1,..., N). \]
Exploiting this mass conservation law, Proposition \ref{miF} below, Strichartz estimates, and blow-up alternative,  we prove the above global existence.

\begin{Remark} The sign of  $\kappa$  in Hartree and Fock terms determines the defocusing and  focusing character of the nonlinearity. We shall see that  this  will play no role in our analysis, as we do not use the  conservation  of energy of $(\#)$ to achieve global existence. This  contrasts with  well-posedness scenario  in $H^{1/2}(\R^3)$. For example, in \cite[Theorem 2.3]{frohlich2007dynamical}  it is proved that radially symmetric data with negative energy lead to blow-up solutions in finite time for \eqref{rhf} with $\alpha=\gamma=1$. 

For higher order Hartree-Fock equation Carles et. al. in \cite[Corollary 4.7]{carles2015higher}  obtained propagation of $H^s(\R^3)$-regularity with $s\in\mathbb{N}$.  We  plan to address similar result in the spaces involving $L^p(\rd)$-integrability.
\end{Remark}
We now turn our attention to the well-posedness of  $(\#)$ in the Fourier-Lebesgue spaces $\widehat{L}^p(\rd)$   (with $1\leq p \leq \infty$) defined by  
\begin{equation*}\label{hs}
\widehat{L}^p(\rd)=\big\{f\in\mathcal{S}'(\rd): \n f\n_{\widehat{L}^p}:=\n\F f\n_{L^{p'}} < \infty \big\}
\end{equation*}
where $\frac{1}{p}+\frac{1}{p'}=1.$ 
We note that by Hausdroff-Young inequality,  $L^p(\R^d) \subset \widehat{L}^p(\R^d)$ for $1\leq p \leq 2$ and $ \widehat{L}^p(\R^d) \subset L^p(\R^d)$ for $2\leq p \leq \infty.$ We denote by $L^2_{rad}(\R^d)$,  space of radial functions in  $L^2(\R^d)$. Now we state the following theorem:
\begin{theorem}[Local well-posedness in $ \widehat{L}^p\cap L^2$]\label{lhw}
 Let 
\begin{equation*}
X= \begin{cases}  \widehat{L}^p(\rd)\cap L^2(\rd) & \text{if} \  a\geq0,\ \alpha=2,\ 0<\gamma<\min\{2,\frac{d}{2}\},\  p\in [1, \infty]\\
\widehat{L}^p(\rd) \cap L_{rad}^{2}(\rd) & \text{if} \  a\geq0,\ d\geq2,\ \frac{2d}{2d-1}<\alpha<2, \ 0<\gamma<\min\{\alpha,\frac{d}{2}\},\  p\in [1, \infty]\\
\widehat{L}^p(\rd)\cap L^2(\rd) & \text{if} \ a=0,\ 0<\alpha<\infty,\ 0<\gamma<   2d(\frac{1}{2}-\frac{1}{p}) , p\ \in (2, \infty]\\
\widehat{L}^p(\rd) & \text{if}\  a>0,\ 0<\alpha<\infty,\ 0<\gamma<  2d(\frac{1}{2}-\frac{1}{p}),  p\ \in (2, \infty].
\end{cases}
\end{equation*}
Assume that  $(\psi_{0,1}, ..., \psi_{0,N}) \in X^N.$ Then there  exist $T>0$ and a unique local solution $ (\psi_1,..., \psi_N)$ to $(\#)$ such that
 \[\left(\psi_1(t), ..., \psi_N(t)\right) \in  \left(C([0, T], X\right))^N.\]
  Moreover,  the map  $(\psi_{0,1}, ..., \psi_{0,N}) \mapsto  (\psi_{1}, ..., \psi_{N})$ is locally Lipschitz from   $\widehat{L}^p(\mathbb R^d)\cap L^2(\rd)$ to   $\big(C([0, T], \widehat{L}^p (\R^d)\cap L^2(\rd))\big)^N.$ 
\end{theorem}

The fractional Schr\"odinger propagator $U_{\alpha}(t):=e^{-it(-\Delta)^{\alpha/2}}$ is bounded in $\widehat{L}^p(\R^d)$ for all $1\leq p \leq \infty$ (see Lemma \ref{4}). Hence, we do not need to transfer $(\#)$ using the semigroup  $U_{\alpha}(-t)$  to establish  local existence in $\widehat{L}^p(\rd)\cap L^2(\rd).$ This  contrasts  with local solutions of Theorem \ref{lw}.  In \cite[Proposition 3.3]{carles2014cauchy}, Carles-Mouzaoui   proved  that  the Hartree equation is locally well-posed in $\widehat{L}^{p}(\rd)\cap L^2(\rd)$ for $p=\infty$  and Bhimani \cite[Proposition 4.5]{Bhimani2019global}  proved this result for the fractional Hartree equation. 
Hyakuna \cite[Theorem 1.8]{hyakuna2018global} proved local well-posedness for the Hartree equation in $\widehat{L}^{p}(\rd)\cap L^2(\R^d)$ for $2\leq p \leq \infty.$  In fact, he used Zhou spaces to prove local existence. The particular case of Theorem \ref{lhw} establishes these results for any  $1\leq p \leq \infty.$   
In  \cite[Theorem  1.2]{hyakuna2018multilinear}   it is proved that the Hartree equation is locally well-posed  in $\widehat{L}^{p}(\rd) \ (2\leq p<\infty) $  if $2d(\frac{1}{2}-\frac{1}{p})\leq\gamma<\min\{2,d\}$. The particular case of Theorem \ref{lhw} reveals that  this result is also true  for any $2\leq p \leq \infty$ and $0<\gamma<  2d(\frac{1}{2}-\frac{1}{p})$  with Yukawa type self interaction.

\begin{Remark}  
\begin{enumerate}
\item  In Theorem \ref{lhw} the radial symmetry assumption  for initial  data comes due to use of fractional Strichartz estimate (Proposition \ref{fst}) in the proof.

 \item  We have the local existence result  (Theorem \ref{lhw}) in $\widehat{L}^p(\R^d)$ for $ 2<p \leq \infty$, without any radial assumption  on initial data, if $0<\gamma<2d(\frac{1}{2}-\frac{1}{p}).$  
 The analogue of this  result  for $1\leq p<2$ is not known to us.
\end{enumerate}
\end{Remark}


\begin{theorem}[Global well-posedness in $L^2 \cap \widehat{L}^p$]\label{global2}
Let 
\begin{equation*}
X= \begin{cases}  \widehat{L}^p(\rd)\cap L^2(\rd) & \text{if} \  a\geq0,\ \alpha=2, 0<\gamma<\min\{2,\frac{d}{2}\},\  p\in [1, \infty]\\
\widehat{L}^p(\rd) \cap L_{rad}^{2}(\rd) & \text{if} \  a\geq0,\ d\geq 2,\ \frac{2d}{2d-1}<\alpha<2, 0<\gamma<\min\{\alpha,\frac{d}{2}\},\  p\in [1, \infty].
\end{cases}
\end{equation*}
Assume  that $(\psi_{0,1}, ..., \psi_{0,N}) \in X^N$. Then the local solution to $(\#)$ given by Theorem \ref{lhw} extends to a global one such that 
\[\left (\psi_1(t),...,\psi_N(t)\right) \in \big( C(\R,X)\cap L^{4\alpha/\gamma}_{loc}(\mathbb R, L^{4d/(2d-\gamma)}(\rd) )  \big)^N.\]
\end{theorem}

Carles-Mouzaoui \cite[Theorem 1.1]{carles2014cauchy}  proved  that the Hartree equation is globally well-posed in $L^2(\rd)\cap\widehat{L}^\infty(\rd)$ and  Hyakuna \cite[Theorem 1.9]{hyakuna2018global}  generalized  this result to $L^2(\rd)\cap\widehat{L}^{p}(\rd)$ for $2\leq p < \infty.$ 
 On the other hand, Bhimani \cite[Theorem 1.2]{Bhimani2019global}  generalized the Carles-Mouzaoui result to the  fractional  Hartree equation in $L^2(\rd)\cap\widehat{L}^{\infty}(\rd).$  The particular case of  Theorem \ref{global2} establishes these results for any  $1\leq p \leq \infty.$

\begin{Remark}
\begin{enumerate}
 \item To extend the local existence (proved in  $\widehat{L}^p(\R^d)$  for $\alpha \neq 2$) globally, first we prove that $(\#)$ is globally well-posed  (see Proposition \ref{miF})  in  $L_{rad}^2(\mathbb R^d)$  via  Strichartz estimates  for the fractional Sch\"odinger equation (see Proposition \ref{fst})  $-$where we need initial data to be a radial, $\alpha \in (2d/(2d-1), 2)$  and   dimension $d\geq 2$ (see \cite[p.26-27]{guo2014improved}). Invoking Proposition \ref{miF}, we get the global existence in  $\widehat{L}^p(\rd) \cap L^2(\rd)$. Thus we notice  in the proof that, to take  the advantage of Proposition \ref{miF},   the hypothesis,  initial data to be radial  of  Theorem \ref{global2} is necessary.

\item    The analogue for Theorem \ref{global2}  without radial assumption on initial data remains an interesting open question.

\item Due to lack of appropriate Strichartz estimates for  the fractional Schr\"odinger equation with $\alpha>2,$ we do not know  whether  $(\#)$ with $\alpha>2$ is globally well-posed in $L^2(\mathbb R^d)$ and also whether the solution is in some mixed $L_{loc}^{p(\gamma)}(\mathbb R, L^{q(\gamma)}(\mathbb R^d))$ spaces (the analogue of Proposition \ref{miF}).  In view of this, the analogue of Theorem \ref{global2}  for $\alpha>2$ remains another  open  question.
\end{enumerate}
\end{Remark}




\begin{theorem}[Improved well-posedness  in 1D]\label{iglobal} Let $\alpha=2, 0<\gamma<1,$ and 
\begin{equation*}
X=\begin{cases}  L^p(\R)\cap L^2(\R)  & \text{if} \ p\in(4/3,2]\\
\widehat{L}^p(\R)\cap L^2(\R) & \text{if}  \ p\in[2, 4).
\end{cases}
\end{equation*}
Assume  that  $\psi_0= (\psi_{0,1}, ..., \psi_{0,N})\in X^N.$
Then  there exists a unique global solution  of $(\#)$ such that 
$\left (U(-t)\psi_1(t),..., U(-t)\psi_N(t)\right) \in \left( C(\R,X) \right)^N$ when $p\in (4/3, 2]$ and $ \left (\psi_1(t),..., \psi_N(t)\right) $ $ \in \left( C(\R, X) \right)^N$ when $p\in [2,4).$
\end{theorem}
The proof of Theorem \ref{iglobal} relies on generalized Strichartz estimates (see Lemma \ref{31})  in  1D and following the strategy in \cite{hyakuna2018global}. Specifically,  we shall see that this enables us to estimate   the integral nonlinear part of $(\#)$ (see  \eqref{nl} and \eqref{e1}),  and as a consequence we can improve the range of $\gamma.$

   We next show that  $(\#)$  with Coulomb type potential shows a mild form of  ill-posedness
   in the mere $\widehat{L}^p(\R^d)^N \ (2<p\leq\infty)$ spaces for $0<\gamma<2d(\frac{1}{2}-\frac{1}{p})$. Specifically, we have the following result:

\textcolor{black}{
\begin{theorem}[Failure of $C^3-$smoothness in  $\widehat{L}^p$
]\label{ill1}
Let $a=0$, $0<\gamma<2d(\frac{1}{2}-\frac{1}{p}),$ $2<p\leq\infty$ and fix $0<t\leq T.$
Denote the solution map of (\#) by $\mathcal{U}(t):\psi_0\mapsto\psi(t).$ Assume that $\mathcal{U}(t)$ is well-defined as a  map acting in $(\widehat{L}^p(\R^d))^N,$ then $\mathcal{U}(t)$ is not $C^3-$smooth at $\psi_0=0$ in $(\widehat{L}^p(\R^d))^N.$
\end{theorem}
The study of mild ill-posedness (failure of $C^3-$smoothness) for the Hartree-Fock equation is new as far as we are aware.   This type of mild ill-posedness was initiated by Bourgain in \cite{bourgain1997periodic} for KdV and mKdV,  see also \cite{tzvetkov1999remark}. This essentially involves showing ``unboundedness" of  the third  Picard iterate  associated with (\#), see \eqref{bi} and Subsection \ref{fr}.  Since then,  many authors have used this approach,  see,  for example,  \cite[Proposition 4.1]{choffrut2018ill} for cubic nonlinear half-wave equation. We exploit this approach to prove Theorem \ref{ill1}.  See point \eqref{rah} below for further comments on the proof.
}

\textcolor{black}{
Theorem \ref{ill1} states that (\#) experience a mild form of ill-posedness in the sense that the solution map fails to be $C^3-$smooth.  However,  we do not know whether the solution map fails to be continuous at origin in $(\widehat{L}^p(\R^d))^N$. This remains an interesting open question.  On the other hand,  we  could  show  that (\#) is not quantitatively well-posed 
  in $(\widehat{L}^p(\R^d))^N$.    See Subsection \ref{fr} for detailed discussion.}

We summarize our findings in Table \ref{table}. We write $x\wedge y=\min\{ x, y \}$. \\
\begin{table}
\caption{\textbf{Results Summary}}
\label{table}
\begin{tabular}{|P{1.1cm}|P{1.4cm}|P{1.7cm}|P{1.6cm}|P{3.1cm}|P{2.8cm}|}
\hline
$a$	 & $\alpha$	&			\textbf{Space}		&		$p$	 &		$\gamma$ 		&\textbf{Result}\\
\hline
\multirow{11}{*}{$[0,\infty)$}		&	\multirow{8}{*}{$\{2\}$}	&			\multirow{4}{*}{$L^p\cap L^2$}					& \addstackgap[4pt]{$[1,\infty]$}	&$ (0, 1\wedge\frac{d}{2})$&		$global$\\
\cline{4-6}
		   &		    		 &				  	    &		\addstackgap[4pt]{$[1,\frac{4}{3}]$	}	&$ (0, 1\wedge 2d(\frac{1}{p}-\frac{1}{2}))$&	{$local$}\\
		   \cline{4-6}
&		&		&		   \addstackgap{$(\frac{4}{3},2]$}  & $(0,1)$ & $global$ $if$ $d=1$\\
\cline{3-6}
	&			&	\multirow{4}{*}{$\widehat{L}^p\cap L^2$}	&	\addstackgap[4pt]{$[1,\infty]$}	&	$(0,2\wedge\frac{d}{2})$	&	{$global$}\\
\cline{4-6}
		   &		    		 &				  	    &		\addstackgap[4pt]{$(2,\infty]$}		&	$(0,2d(\frac{1}{2}-\frac{1}{p}))$& $local$\\
\cline{4-6}
		   &		    		 &				  	    &		{$[2,4)$}	& 	\addstackgap[4pt]{$(0,1)$}	&$global$ $if$ $d=1$ \\
		   \cline{2-6}
		   &\addstackgap[4pt]{$(0,\infty)$}	& 	{$\widehat{L}^p\cap L^2$} &  $(2,\infty]$&$(0,2d(\frac{1}{2}-\frac{1}{p}))$& $local$\\
		   \cline{2-6}
		   &		\addstackgap{$(\frac{2d}{2d-1},2)$}    		 &		{$\widehat{L}^p\cap L_{rad}^2$}	&	{$[1,\infty]$}	&		{$(0,\alpha\wedge\frac{d}{2})$}	&		{$global$ $if$ $d\geq2$}\\
\hline
\addstackgap[4pt]{$(0,\infty)$}   &	$(0,\infty)$	&	{$\widehat{L}^p$} & $(2,\infty]$		& $(0,2d(\frac{1}{2}-\frac{1}{p}))$ & {$local$}\\
\hline
\addstackgap[4pt]{$\{0\}$}	&		$(0,\infty)$	&	{$\widehat{L}^p$} & $(2,\infty]$		& $(0,2d(\frac{1}{2}-\frac{1}{p}))$ 	& $mild$ $ill$-$posed$\\
\hline
\end{tabular}
\end{table}

\subsection{Several comments on  the paper}\label{SCP} 
\begin{enumerate}
\item \label{JPE} We define \textbf{trilinear operator}  associated to Hartree-type nonlinearity  by 
\[ \widehat{\h}_{a,\gamma,t}(f,g,h)=\left[\left(S_{a,t}*|\cdot|^{\gamma-d}\right)(f*\bar{g})\right]*h\]
where 
\begin{equation}\label{si}
S_{a,t}=\begin{cases}
 \ \ \ \ \ \ \ \ \ \delta_0 & \text{if} \ a=0\\
\ \frac{a|t|}{(4a^2t^2+|\cdot|^2)^{(d+1)/2}} & \text{if} \ a>0
\end{cases}
\end{equation}
with $t\in\R$ and $\delta_0$ being  the Dirac distribution with mass at origin in $\rd$. We shall briefly discuss our main  ideas and techniques to establish local well-posedness of $(\#)$ in $L^p(\rd)\cap L^2(\rd)$ for $2<p \leq \infty:$
\begin{itemize}
\item  The proof of local well-posedness  crucially depends on a suitable availability of   $\widehat{H}_{a, \gamma, t}-$ estimates. (See also  Remark \ref{mip} below.) Specifically,  we have following triliner estimates (to be established in Subsection \ref{TE} below):
\begin{equation}\label{AE}
\|\widehat{\h}_{a,\gamma,t}(f,g,h)\|_{L^p\cap L^2} \lesssim \|f\|_{L^p \cap L^2} \|g\|_{L^p \cap L^2} \|h\|_{L^p \cap L^2}, \quad \forall p \in [1, 2).
\end{equation}
\item Using \eqref{AE} one can establish local well-posedness in $L^p(\rd)\cap L^2(\rd)$ for $1\leq p <2.$ Up to now we cannot know the validity of  \eqref{AE} for $p\in (2, \infty].$ Hence, we cannot follow previously employed ideas to deal with the case $2<p\leq \infty.$ 
\item  To overcome the issue  mentioned above,  we {\bf introduced  a ball} $\mathcal{V}_b^T$, (see Case I in the first Proof of Theorem \ref{lw} in Subsection \ref{plw}) {\bf  involving the ``twisting"  free Schr\"odinger propagator}    (unlike the usual  choice).  We then effectively  use Strichartz estimates  to  establish the contraction of  twisted Duhamel operator $\Phi_{\psi_0}$  (see \eqref{6}) on $\mathcal{V}_b^T.$ This leads to local existence.   The advantage of this approach is that, though we do not know \eqref{AE} for $2<p\leq\infty$, we could establish local well-posedness for $2<p\leq \infty.$ In fact, this approach works for all $p\in [1, \infty].$
\end{itemize}

\item \label{dsv} It might be tempting to think that well-posedness for $(\#)$ with  Coulomb type self-interaction  immediately implies well-posdness for $(\#)$ with  Yukawa type self interaction  as  $e^{-a|\cdot|}|\cdot|^{-\gamma}$ (Yukawa type: $a>0$) has faster decay at infinity compared to $|\cdot|^{-\gamma}$ (Coulomb type: $a=0$).  However, it is not the case. The local existence in $L^p(\rd)\cap L^2(\rd)$ heavily rely on  factorization Lemma \ref{2}. Besides we shall  notice that for $a>0$ the trilinear operator $\widehat{H}_{a,\gamma,t}$  depends on a time parameter $t$. On the other hand, $\widehat{H}_{a,\gamma,t}$  is independent of time parameter   in the case of $a=0$. Specifically, we cannot straight-way claim $$\n\widehat{\mathcal{H}}_{a,\gamma,t}(f,g,h)\n_{L^p\cap L^2}\leq c \n\widehat{\mathcal{H}}_{0,\gamma,t}(f,g,h)\n_{L^p\cap L^2}$$ and apply estimates for $a=0$.  This new time parameter makes  analysis more delicate while dealing with Yukawa type self interaction. (Also, cf. Theorems \ref{lhw} and \ref{ill1} for $a=0$ and $a>0.$)

\textcolor{black}{
\item \label{rah} In order to prove  unboundedness of third Picard iterate in the proof of Theorem \ref{ill1}, we have carefully adopted the technique form \cite{carles2014cauchy} (where the unboundedness is proved for classical Hartree equation  in $\widehat{L}^{\infty}$).  In  \eqref{hf} due to the presence of  Fock term (exchange term) $ F_\psi(\psi_k)$,  the   computation of the third iterate is more subtle compared to the classical Haretree case and  thus  required careful analysis to prove Theorem \ref{ill1}.    We also note that the proof of Theorem \ref{ill1} relies on the fact that  the Fourier transform of the Coulomb type potential is homogeneous. On the other hand, the Fourier transform of the Yukawa type potential is not homogeneous, see Lemma \ref{1}. Therefore the proof does not work for Yukawa type potential.
}

\item   For $s>0$ we can choose $2<p\leq\infty$ so that $s>d(\frac{1}{2}-\frac{1}{p}).$  In this case, by H\"older inequality,  we have $H^s(\rd)\subset \widehat{L}^p(\rd)$ and so $H^s(\rd)\subset \widehat{L}^p(\rd) \cap L^2(\rd) \subset L^2(\rd).$  Thus Theorems \ref{lhw} and \ref{global2}  reveals that we can solve $(\#)$ with Cauchy data  in the larger space $\widehat{L}^p(\rd)\cap L^2(\rd)$  compared to $H^{s}(\rd).$  In particular, this complements  Lenzmann's  work \cite[Theorems 1, 2 and 3]{lenzmann2007well} on  Hartree equation with Yukawa type self interaction, especially for $0<s<1/2$ in dimension 3 and for all $s>0$ in other dimensions. 

\item \label{dv} 
We note that
\begin{equation*}
\begin{rcases}
\hspace{1.25cm} H^{1/2}(\mathbb R^3)\\
\widehat{L}^{6}(\mathbb R^3)\cap L^2(\mathbb R^3)
\end{rcases}\subset  L^6(\mathbb R^3) \cap L^2(\mathbb R^3).
\end{equation*}
In  \cite[Theorem 2.3]{frohlich2007dynamical} Fr\"ohlich-Lenzmann  proved that radially symmetric data with negative energy lead to blow-up solutions for $(\#)$ with $\gamma=\alpha=1$  and $a=0$ in finite time in  $H^{1/2}(\mathbb R^3)-$norm. On the other hand, Theorem \ref{global2} ensures that 
$(\#)$ with $3/5< \alpha < 2, 0< \gamma< \min\{ \alpha, 3/2\}$ is globally well-posed in $\widehat{L}^6(\mathbb R^3)\cap L^2(\mathbb R^3).$
\end{enumerate}
 
This paper is organized as follows. In Section \ref{pki}, we introduce preliminaries and establish key ingredients which will be used in the sequel. Specifically, in Subsections \ref{fls} and \ref{TE} we prove  factorization Lemma \ref{2} and  various trilinear estimates respectively. We shall see this will play a vital role in proving main results of the paper.  In  Subsections \ref{plw}, \ref{plhw} we prove  Theorems \ref{lw} and \ref{lhw} respectively.
In Subsections \ref{pglobal1}, \ref{pglobal2} we prove  Theorems \ref{global1} and \ref{global2} respectively. In Subsection \ref{piglobal}, we prove Theorem \ref{iglobal}. In Section \ref{pill1} we prove Theorem \ref{ill1}.
\section{Preliminaries and key ingredient}\label{pki}
\noindent
\textbf{Notations and known results}.  For real numbers $A,B$ the notation $A \lesssim B $ means $A \leq cB$ for some universal constant $c > 0 $, whereas $ A \asymp B $ means $c^{-1}A\leq B\leq cA $ for some $c\geq 1$. Also $A\gtrsim B$ means $B\lesssim A$. For complex $A,B$ by $A\propto B$ means $A=cB$ for some universal constant $c \neq 0 $.
The characteristic function of a set $E\subset \mathbb R^d$ is $\chi_{E}(x)=1$ if $x\in E$ and $\chi_E(x)=0$ if $x\notin E.$
Let $I\subset \mathbb R$ be an interval and $X$ be a Banach space of functions. Then the norm of the space-time Lebesgue space $L^q(I, X)$ is defined by
$\|u\|_{L^q(I, X)}= \left(\int_{I} \|u(t)\|_{X}^q dt \right)^{1/q}$
and when $I=[0,T]$, $T>0$ we denote $L^q(I, X)$ by $L_T^q(X)$.  For  $p\in [1, \infty],$ its H\"older conjugate, denoted by $p'$,  is given by $\frac{1}{p}+ \frac{1}{p'}=1.$
 The norm  on $N$-fold product $X^N$ of Banach space $\left(X, \|\cdot\|_{X}\right)$ is given by 
$$\n\psi\n_{X^N}=\max_{1\leq j\leq N}\n\psi_j\n_X ,\quad\psi=(\psi_1,\cdots,\psi_N)\in X^N.$$
The Schwartz space is denoted by  $\mathcal{S}(\mathbb R^{d})$ (with its usual topology), and the space of tempered distributions is  denoted by $\mathcal{S'}(\mathbb R^{d}).$ For two Banach spaces of functions $A,B$ in $\mathcal{S}'(\rd)$ we note that $A\cap B$ is also a Banach space with the norm $\n\cdot\n_{A\cap B}=\max\{\n\cdot\n_A,\n\cdot\n_B\}$.
For $x=(x_1,\cdots, x_d), y=(y_1,\cdots, y_d) \in \mathbb R^d, $ we put $x\cdot y = \sum_{i=1}^{d} x_i y_i.$
Let $\mathcal{F}:\mathcal{S}(\mathbb R^{d})\to \mathcal{S}(\mathbb R^{d})$ be the Fourier transform  defined by  
$
\mathcal{F}f(\xi)=\widehat{f}(\xi)=\int_{\mathbb R^{d}} e^{- 2\pi i x\cdot \xi}f(x)dx ,  \xi\in \mathbb R^d.$
Then $\mathcal{F}$ is a bijection  and the inverse Fourier transform  is given by $\mathcal{F}^{-1}f(x)=f^{\vee}(x)=\F f(-x)$ for $x\in\rd$,
and this Fourier transform can be uniquely extended to $\mathcal{F}:\mathcal{S}'(\mathbb R^d) \to \mathcal{S}'(\mathbb R^d)$ such that for each $u\in\mathcal{S}'(\rd)$ one has $\langle \F u,\varphi\rangle=\langle u,\F\varphi\rangle$ for all $\varphi\in\mathcal{S}(\rd)$. 
\begin{Lemma}
\label{1} \begin{enumerate}
\item  
For $f(x)= e^{-2\pi |x|},$  we have 
$ \widehat{f}(\xi) = \frac{c_d}{(1+ |\xi|^2)^{(d+1)/2}}$
with $c_d=\frac{\Gamma((d+1)/2)}{\pi ^{(d+1)/2}},$ $\Gamma$ is the Gamma function and $\widehat{f}\in L^p(\mathbb R^d)$ for all $1\leq p \leq \infty.$

\item  Let $0<\gamma <d.$ Then for $f(x)=|x|^{-\gamma},$ we have 
$
\widehat{f}(\xi)= \frac{ C_{d, \gamma}}{|\xi|^{d-\gamma}}.
$
\end{enumerate}
\end{Lemma} 

For $f\in \mathcal{S}(\mathbb R^{d}),$ we define the \textbf{fractional Schr\"odinger propagator $e^{-it(-\Delta)^{\alpha/2}}$ }for $t \in \mathbb R, \alpha >0$ as follows:
\begin{eqnarray}
\label{sg}\label{-1}
[U_{\alpha}(t)f](x)=\left[e^{-it (-\Delta)^{\alpha/2}}f\right](x):= \int_{\mathbb R^d}  e^{-(2\pi)^\alpha i |\xi|^{\alpha}t}\, \widehat{f}(\xi) \, e^{2\pi i \xi \cdot x} \, d\xi.
\end{eqnarray}
For  $\alpha=2,$ we  simply  write $U_2=U$. In this case we have (see \cite[Lemma 2.2.4]{cazenave2003semilinear}) \begin{align}\label{0}
[ U(t)f](x)=\left[e^{it \Delta}f\right](x)=\frac{1}{(4\pi it)^{d/2}}\int_{\rd}e^{i|x-y|^2/4t}f(y)dy.
 \end{align}

\begin{Definition} A pair $(q,r)$ is $\alpha$-fractional admissible if  $q\geq 2, r\geq 2$ and
$\frac{\alpha}{q} =  d \left( \frac{1}{2} - \frac{1}{r} \right),(q,r,d)\neq(\infty,2,2).$
\end{Definition}
\begin{proposition}[Strichartz estimates] \label{fst}
Denote
$$DF(t,x) =  U_{\alpha}\phi(x) +  \int_0^t U_\alpha(t-s)F(s,x) ds.$$
\begin{enumerate}
\item  \label{fst1} Let $\phi \in L^2(\rd),$ $d\in \mathbb N$ and $\alpha=2.$   Then for any time interval $I\ni0$ and 2-admissible pairs $(q_j,r_j)$, $j=1,2,$ 
there exists  a constant $C=C(r_1,r_2)$ such that 
$$ \|D(F)\|_{L^{q_1}(I,L^{r_1})}  \leq  C \|\phi \|_{L^2}+   C  \|F\|_{L^{q'_2}(I,L^{r'_2})}, \quad\forall F \in L^{q_2'} (I, L^{r_2'}(\R^d))$$  where $q_j'$ and $ r_j'$ are H\"older conjugates of $q_j$ and $r_j$
respectively \cite{keel1998endpoint}.

\item \label{fst2} Let  $\phi \in L_{rad}^2(\rd)$, $d\ge 2,$ and  $\frac{2d}{2d-1} < \alpha < 2.$   Then for any time interval $I\ni0$ and $\alpha$-fractional admissible pairs $(q_j, r_j)$, $j=1,2,$ 
there exists  a constant $C=C(r_1, r_2)$ such that
$$ \|D(F)\|_{L^{q_1}(I,L^{r_1})}  \leq  C \|\phi \|_{L^2}+   C  \|F\|_{L^{q'_2}(I,L^{r'_2})}, \quad \forall F \in L^{q'_2}(I,L_{rad}^{r'_2}(\R^d))$$  where $q_j'$ and $ r_j'$ are H\"older conjugates of $q_j$ and $r_j$ respectively \cite[Corollary 3.10]{guo2014improved}.
\end{enumerate}
\end{proposition}
For the sake of completeness,  we recall the  following standard existence result.  We shall see that  this result will play a vital role to  prove  global existence (Theorems \ref{global1}, \ref{global2}, and \ref{iglobal}). Specifically,  we have the following:
\begin{proposition}\label{miF}  Let $ \alpha>0,$  $0 < \gamma < \min\{\alpha, d\}$ and 
\begin{equation*}
X=\begin{cases} L^2(\R^d) & \text{if} \   \alpha =2, d\geq 1\\
 L_{rad}^2(\R^d) & \text{if} \   \frac{2d}{2d-1} < \alpha < 2, d\geq 2.
\end{cases}
\end{equation*}
If $ (\psi_{0,1}, ...., \psi_{0,N}) \in X^N$ then $(\#)$ has a unique global solution 
$$ (\psi_1,...,\psi_N)   \in  \big(C(\mathbb R, L^2(\rd))\cap L^{4\alpha/\gamma}_{loc}(\mathbb R, L^{4d/(2d-\gamma)}(\rd))\big)^N.$$ 
In addition, its $L^{2}$-norm is conserved, 
$\|\psi_k(t)\|_{L^{2}}=\|\psi_{0,k}\|_{L^{2}}, \   \forall t \in \mathbb R, k =1,2,...,N$
and for all $\alpha$-fractional admissible pairs  $(q,r),$ one has $  (\psi_1,..., \psi_N) \in  \left( L_{loc}^{q}(\mathbb R, L^r(\rd)) \right)^{N}.$ 
\end{proposition}
\begin{proof}[{\bf Proof}] For the proof  of  case $a=0,$ that is, $(\#)$ with  Coulomb  type potential,  see \cite[Propositions  4.2 and 4.3]{bhimani2020hartree}.  The proof of case $a>0,$ that is, $(\#)$ with Yukawa type potential is similar. Hence, we  omit the details. See also \cite[Proposition 2.3]{carles2014cauchy} and \cite[Theorem 4.9]{carles2015higher}.
\end{proof}
\subsection{Factorization formula  for Schr\"odinger propagator}\label{fls}
For $t\neq 0,$ we define  multiplication, dilation and reflection operators (for functions  $w$ on $\R^d$)  and their inverses as follows:
\begin{itemize}
\item  multiplication: $M_t w(x)=e^{i|x|^2/4t}w(x), M_t^{-1}w(x)= e^{-i|x|^2/4t} w(x)$
\item dilation: $D_tw(x)=\frac{1}{(4\pi it)^{d/2}}w \left(\frac{x}{4\pi t}\right)$  and $D_t^{-1}w(x)={(4\pi it)^{d/2}}w \left({4\pi t x}\right)$
\item reflection:  $Rw(x)=w(-x)$ and  $R^{-1} w(x)= w(-x).$
\end{itemize}

\begin{Lemma}\label{fl}
 Let $0\neq t \in \R$ and $\varphi \in \mathcal{S}(\R^d).$ Then we have  
$U(t)\varphi=M_tD_t\mathcal{F}M_t\varphi$ and $U(-t)\varphi=M_t^{-1}\F^{-1}D_t^{-1}M_t^{-1}\varphi.$
\end{Lemma}
\begin{proof}[{\bf Proof}]
Follows from formula \eqref{0} and the above  definitions,
see \cite[p.372]{hayashi1998asymptotics}. 
\end{proof}

Now, for $f,g, h \in \mathcal{S}(\R^d), a\geq 0, t\in \R,$ and $0<\gamma<d,$ we define \textbf{trilinear operators}  associated to Hartree-type nonlinearity  as follows
\begin{equation}\label{tf}
\h_{a,\gamma}(f,g,h)=\left[\frac{e^{-a|\cdot|}}{|\cdot|^\gamma}*(f\bar{g})\right]h,\quad
\widehat{\h}_{a,\gamma,t}(f,g,h)=\left[\left(S_{a,t}*|\cdot|^{\gamma-d}\right)(f*\bar{g})\right]*h
\end{equation}
where $S_{a,t}$ is as in \eqref{si}.
Now we decompose $\widehat{\h}_ {a,\gamma, t}$  in the following way 
\begin{equation}\label{dtf}
\widehat{\h}_{a,\gamma,t}^j(f,g,h):=\left[\left(S_{a,t}*k_j\right)(f*\bar{g})\right]*h,\quad  j=1,2
\end{equation}
where $k_1,k_2$ are given by
\begin{equation}\label{di}
k_1(x)= \chi_{\{|x|\leq1\}}(x)|x|^{\gamma-d},\quad
k_2(x)=\chi_{\{|x|>1\}}(x)|x|^{\gamma-d}.
\end{equation}
Note that $k_1 \in L^p(\R^d)$ for $1\leq p < \frac{d}{d-\gamma}$,  $k_2\in L^q (\R^d)$ for $\frac{d}{d-\gamma} <q \leq \infty$ and
$ \widehat{ \mathcal{H}}_{a, \gamma, t}= \widehat{ \mathcal{H}}_{a, \gamma, t}^1+ \widehat{ \mathcal{H}}_{a, \gamma, t}^2. $

\begin{Lemma}\label{2}
Let $ 0 \neq t \in \R, 0 < \gamma <d, a\geq 0,$ and $v_j(t)=U(-t)u_j(t) \in \mathcal{S}(\R^d)$  with $j=1,2,3.$ Then we have $U(-t)\h_{a,\gamma}(u_1,u_2,u_3)\asymp|t|^{-\gamma}M_t^{-1}\widehat{\h}_{a,\gamma,t}(M_tv_1,RM_tv_2,M_tv_3).$
\end{Lemma}
\begin{proof}[{\bf Proof}]
For the case $a=0,$ see \cite[Lemma 2.1]{hyakuna2018global}. So it remains to prove the case $a>0$.
Note that $D_t^{-1}(fg)=(4\pi it)^{-d/2}(D_t^{-1}f)(D_t^{-1}g)$, $\F^{-1}D_t^{-1}=D_{-t}\F=cRD_t\F$ and  $U(t)\bar{u}= \overline{U(-t) u}.$
Using these  equalities   and performing  the change of variable, we  may rewrite 
\begin{align*}
& D_t^{-1}\left[\left(|\cdot|^{-\gamma}e^{-a|\cdot|}\right)*(fg)\right](x)=  (4\pi it)^{d/2}\left(\left(|\cdot|^{-\gamma}e^{-a|\cdot|}\right)*(fg)\right)(4\pi tx)\\
&=  (4\pi it)^{d/2}(4\pi t)^d\int_{\rd}|4\pi ty|^{-\gamma}e^{-4a\pi|ty|}(fg)(4\pi t(x-y))dy\\
& = (-4\pi it)^{d/2}(4\pi|t|)^{-\gamma}\left(\left(|\cdot|^{-\gamma}e^{-4a\pi|t\cdot|}\right)*\left(D_t^{-1}fD_t^{-1}g\right)\right)(x).
\end{align*}
Using the above  equalities and  Lemma \ref{fl}, we obtain
\begin{align*}
 &M_tU(-t)\h_{a,\gamma}(u_1,u_2,u_3)=\F^{-1}D_t^{-1}M_t^{-1}\h_{a,\gamma}(u_1,u_2,u_3)\\
&=\F^{-1}D_t^{-1}\left(\left(\left(|\cdot|^{-\gamma}e^{-a|\cdot|}\right)*((M_t^{-1}u_1)(M_t\bar{u_2}))\right)M_t^{-1}u_3\right)\\
&\propto t^{-d/2}\F^{-1}\left[D_t^{-1}\left(\left(|\cdot|^{-\gamma}e^{-a|\cdot|}\right)*((M_t^{-1}u_1)(M_t\bar{u_2}))\right)D_t^{-1}M_t^{-1}u_3\right]\\
&\propto |t|^{-\gamma}\F^{-1}\bigg[\left(\left(|\cdot|^{-\gamma}e^{-4a\pi|t\cdot|}\right)*((D_t^{-1}M_t^{-1}u_1)(D_t^{-1}M_t\bar{u_2}))\right) D_t^{-1}M_t^{-1}u_3\bigg]\\
&\propto|t|^{-\gamma}\big[\left(|\cdot|^{\gamma-d}*\F^{-1}e^{-4a\pi|t\cdot|}\right)((\F^{-1}D_t^{-1}M_t^{-1}u_1) \ *(\F^{-1}D_t^{-1}M_t\bar{u_2}))\big]*\F^{-1}D_t^{-1}M_t^{-1}u_3.
\end{align*}
Since, by Lemma \ref{1},
\begin{align*}
\left(\F^{-1}e^{-4a\pi|t\cdot|}\right)(\xi)&=\frac{c_d}{(2a|t|)^d}\left(1+\frac{|\xi|^2}{4a^2t^2}\right)^{-(d+1)/2}\propto\frac{a|t|}{(4a^2t^2+|\xi|^2)^{(d+1)/2}}=: S_{a,t} \ \  (a>0),
\end{align*}
it follows that 
\begin{align*}
M_tU(-t)\h_{a,\gamma}(u_1,u_2,u_3)&\propto|t|^{-\gamma}\left[\left(|\cdot|^{\gamma-d}*S_{a,t}\right)((M_tU(-t)u_1)*(RD_t\F M_t\bar{u_2}))\right] *M_tU(-t)u_3\\
&\propto|t|^{-\gamma}\big[\left(|\cdot|^{\gamma-d}*S_{a,t}\right)((M_tU(-t)u_1)*(RM_t^{-1}U(t)\bar{u_2}))\big] *M_tU(-t)u_3\\
&\propto|t|^{-\gamma}\big[\left(|\cdot|^{\gamma-d}*S_{a,t}\right)((M_tU(-t)u_1)*(R\overline{M_tU(-t)u_2}))\big]*M_tU(-t)u_3\\
&=|t|^{-\gamma}\widehat{\h}_{a,\gamma,t}(M_tv_1,RM_tv_2,M_tv_3).
\end{align*}
This completes the proof.
\end{proof}
 
%
 
\subsection{Trilinear estimates}\label{TE} 
In this subsection we prove some useful trilinear estimates for $\widehat{\h}_{a,\gamma,t}$ and $\h_{a,\gamma}$  (see \eqref{tf}).  We start with the following:
\begin{Lemma}\label{3}
Assume $0<\gamma<d$. Let $k_{j} \ (j=1,2)$ and $S_{a,t}$ be  given by \eqref{di} and \eqref{si} respectively. Then we  have 
$\n k_1*S_{a,t}\n_{L^{r_1}}\lesssim\n k_1\n_{L^{r_1}}$ and $\n k_2*S_{a,t}\n_{L^{r_2}}\lesssim\n k_2\n_{L^{r_2}} $
for all $r_1\in[1,\frac{d}{d-\gamma})$ and for all $r_2\in(\frac{d}{d-\gamma},\infty]$.
\end{Lemma}
\begin{proof}[{\bf Proof}]
 The case $a=0$ being trivial assume  that $a>0$. Note that for $d=1,$ we have \begin{align*}
\left\n S_{a,t}\right\n_{L^1}=\int_{\rd}\frac{a|t|}{4a^2t^2+|\xi|^2}d\xi&\asymp\frac{1}{a|t|}\int_0^\infty\frac{dr}{1+(r/2a|t|)^2}\asymp\int_0^\infty\frac{ds}{1+s^2}\asymp1.
\end{align*}
 For $d\geq2$, we obtain  
\begin{align*}
\left\n S_{a,t}\right\n_{L^1}
&\asymp \int_0^\infty\frac{a|t|r^{d-1}}{(4a^2t^2+r^2)^{(d+1)/2}}dr\asymp a|t|\int_{4a^2t^2}^\infty\frac{(s-4a^2t^2)^{(d-2)/2}}{s^{(d+1)/2}}ds\leq a|t|\int_{4a^2t^2}^\infty\frac{s^{(d-2)/2}}{s^{(d+1)/2}}ds=1.
\end{align*}
Now Young inequality,  gives the desired inequalities.
\end{proof}
\begin{Remark}
Note that we separate the computation of $L^1$-norm for $S_{a,t}$ in two cases as  the third step in the proof of case $d\geq2$ does not work for $d=1$. 
\end{Remark}

\begin{proposition}[$L^p$-estimates]\label{le} Let $0<\gamma<d$, $f_j\in L^p(\rd)\cap L^2(\rd) \ (j=1,2,3)$ and   $\widehat{\h}_{a,\gamma,t}$, $\widehat{\h}_{a,\gamma,t}^j$  are given by  \eqref{tf} and  \eqref{dtf} respectively.
\begin{enumerate}
\item \label{lea}
Assume that  $1\leq p<2$ and $0<\gamma<2d(\frac{1}{p}-\frac{1}{2})$.  Then we have
\begin{equation*}\n\widehat{\h}_{a,\gamma,t}^j(f_1,f_2,f_3)\n_{L^2}\lesssim
\begin{cases}  \left\n f_1\right\n_{L^2}\left\n f_2\right\n_{L^2}\left\n f_3\right\n_{L^2} & \text{if} \ j=1\\
\left\n f_1\right\n_{L^p}\left\n f_2\right\n_{L^p}\left\n f_3\right\n_{L^2} & \text{if} \ j=2\\
\end{cases}
\end{equation*}
and 
\begin{equation*}
\n\widehat{\h}_{a,\gamma,t}^j(f_1,f_2,f_3)\n_{L^p}\lesssim 
\begin{cases} \n f_1\n_{L^2}\n f_2\n_{L^2}\n f_3\n_{L^p} & \text{if} \ j=1\\ 
\n f_1\n_{L^p}\n f_2\n_{L^p}\n f_3\n_{L^p} & \text{if} \ j=2.
\end{cases}
\end{equation*}
As a consequence, we have  $
\n\widehat{\h}_{a,\gamma,t}(f_1,f_2,f_3)\n_{L^p\cap L^2}\lesssim\prod_{j=1}^3\n f_j\n_{L^p\cap L^2}.$
\item \label{leb}  Assume that  $2<p\leq\infty$ and $0<\gamma<d(\frac{1}{2}-\frac{1}{p})$.  Then we have 
$
\n\widehat{\h}_{a,\gamma,t}(f_1,f_2,f_3)\n_{L^p}\lesssim\left\n f_1\right\n_{L^2}\left\n f_2\right\n_{L^2}\n f_3\n_{L^p\cap L^2}.
$
\end{enumerate}
\end{proposition}
\begin{proof}[{\bf Proof}] 
\eqref{lea} By   Young, H\"older,  Hausdorff-Young inequalities and Lemma \ref{3}, for $1\leq p \leq \infty,$ we have
\begin{align}\label{ns}
\left\n\left[(k_1*S_{a,t})(f_1*f_2)\right]*f_3\right\n_{L^p}&\leq\left\n(k_1*S_{a,t})(f_1*f_2)\right\n_{L^1}\left\n f_3\right\n_{L^p} \leq\left\n k_1*S_{a,t}\right\n_{L^1}\left\n f_1*f_2\right\n_{L^\infty}\left\n f_3\right\n_{L^p} \nonumber \\
&\leq\left\n k_1*S_{a,t}\right\n_{L^1}\left\n\widehat{f_1}\widehat{f_2}\right\n_{L^1}\left\n f_3\right\n_{L^p} \lesssim\left\n f_1\right\n_{L^2}\left\n f_2\right\n_{L^2}\left\n f_3\right\n_{L^p}.
\end{align}
Since 
$ \frac{1}{p/2(p-1)}+ \frac{1}{p/(2-p)}=1,\ \frac{1}{p}+\frac{1}{p}= 1+ \frac{1}{p/(2-p)}\ \text{and}\ \frac{p}{2(p-1)}>\frac{d}{d-\gamma}, $
H\"older and Young  inequalities  and Lemma \ref{3}  imply 
\begin{align*}
\left\n\widehat{\h}_{a,\gamma,t}^2(f_1,f_2,f_3)\right\n_{L^p}&=\left\n\left[(k_2*S_{a,t})(f_1*f_2)\right]*f_3\right\n_{L^p}\leq\left\n(k_2*S_{a,t})(f_1*f_2)\right\n_{L^1}\left\n f_3\right\n_{L^p}\\
&\leq\left\n k_2*S_{a,t}\right\n_{L^{\frac{p}{2(p-1)}}}\left\n f_1*f_2\right\n_{L^{\frac{p}{2-p}}}\left\n f_3\right\n_{L^p}\lesssim\prod_{l=1}^3\left\n f_l\right\n_{L^p}.
\end{align*}
 Similarly,
$\n\widehat{\h}_{a,\gamma,t}^2(f_1,f_2,f_3)\n_{L^2}
\leq\left\n(k_2*S_{a,t})(f_1*f_2)\right\n_{L^1}\left\n f_3\right\n_{L^2}\lesssim\left\n f_1\right\n_{L^p}\left\n f_2\right\n_{L^p}\left\n f_3\right\n_{L^2}.
$\\
\eqref{leb} Since
$ \frac{1}{p}+1 = \frac{1}{2}+ \frac{1}{(2p)/(p+2)}\ \text{and}\ \frac{2p}{p+2}>\frac{d}{d-\gamma}, $
Young inequality and  Lemma \ref{3}  give
\begin{align*}
\n\widehat{\h}_{a,\gamma,t}^2(f_1,f_2,f_3)\n_{L^p}&=\left\n\left[(k_2*S_{a,t})(f_1*f_2)\right]*f_3\right\n_{L^p}\leq \left\n(k_2*S_{a,t})(f_1*f_2)\right\n_{L^{\frac{2p}{p+2}}}\left\n f_3\right\n_{L^2}\\
&\leq \left\n k_2*S_{a,t}\right\n_{L^{\frac{2p}{p+2}}}\left\n f_1*f_2\right\n_{L^\infty}\left\n f_3\right\n_{L^2}\lesssim\prod_{l=1}^3\left\n f_l\right\n_{L^2}.
\end{align*}
Combining   the above inequality with \eqref{ns}, we get the desired estimate.
\end{proof}

\begin{Remark}\label{le'}
By \eqref{ns} it is clear that for all $p\in[1,\infty],$   $\n\widehat{H}^1_{a,\gamma,t}(f_1,f_2,f_3)\n_{L^p} \lesssim\left\n f_1\right\n_{L^2}\left\n f_2\right\n_{L^2}\left\n f_3\right\n_{L^p}$ holds for $\gamma$ in the wider range  $0<\gamma<d$.
\end{Remark}

\begin{proposition}[$\widehat{L}^p$-estimates] \label{lhe}
\begin{enumerate}
\item  \label{lheb}Assume that  $1\leq p<2$ and $0<\gamma<d(\frac{1}{p}-\frac{1}{2})$. Then we have 
$\left\n\h_{a,\gamma}(f_1,f_2,f_3)\right\n_{\widehat{L}^p}\lesssim\n f_1\n_{L^2}\n f_2\n_{L^2}\n f_1\n_{\widehat{L}^p\cap L^2}.$
\item \label{lhec}  Assume that $2<p\leq\infty$ and $0<\gamma<2d(\frac{1}{2}-\frac{1}{p}),$and let\begin{equation*}
X= \begin{cases} L^2(\rd)\cap \widehat{L}^p(\rd) & \text{if} \  a\geq 0\\
\widehat{L}^p(\rd) & \text{if} \ a>0.
\end{cases}
\end{equation*}
Then we have
$\left\n\h_{a,\gamma}(f_1,f_2,f_3)\right\n_{X}\lesssim\prod_{j=1}^3\left\n f_j\right\n_{X}.$
\end{enumerate}
\end{proposition}
\begin{proof}[{\bf Proof}]
\eqref{lheb} Since $\left\n\h_{a,\gamma}(f_1,f_2,f_3)\right\n_{\widehat{L}^p}=\left\n\F\h_{a,\gamma}(f_1,f_2,f_3)\right\n_{L^{p'}}\asymp\n\widehat{\h}_{a,\gamma,\frac{1}{4\pi}}(\widehat{f_1},\widehat{f_2},R\widehat{f_3})\n_{L^{p'}}$
 using Proposition \ref{le} \eqref{leb} we have 
\begin{align*}
\left\n\h_{a,\gamma}(f_1,f_2,f_3)\right\n_{\widehat{L}^p}\lesssim\n\widehat{f_1}\n_{L^2}\n\widehat{f_2}\n_{L^2}\n\widehat{f_3}\n_{L^2\cap L^{p'}}\asymp\n f_1\n_{L^2}\n f_2\n_{L^2}\n f_1\n_{\widehat{L}^p\cap L^2}.
\end{align*}
\eqref{lhec}  Taking Proposition \ref{le}\eqref{lea} into account, and exploiting  the proof of  Proposition \ref{lhe}\eqref{lheb}, the assertion follows   when $X= L^p(\R^d)\cap L^2(\R^d).$

Next we assume that $X=\widehat{L}^p(\R^d).$  Set $K=e^{-a|\cdot|}|\cdot|^{-\gamma}$.  
Then  $\F K=k_1*h_a+k_2*h_a$ with $h_a(\xi)=\F e^{-a|\cdot|}= \frac{c_da}{(a^2+4\pi^2 |\xi|^2)^{(d+1)/2}}$ (see  \eqref{di} and Lemma \ref{1}) and so it follows that   $\F K\in L^r(\rd)$ for all $\frac{d}{d-\gamma}<r\leq\infty$. 
Since
$\frac{1}{p/2}+ \frac{1}{p'/(2-p')}=1,\ \frac{1}{p'}+ \frac{1}{p'}= 1+\frac{1}{p'/(2-p')},\ p'\leq 2 \ \text{and} \ \frac{p}{2}>\frac{d}{d-\gamma}, $
H\"older and Young inequalities imply
\begin{align*}
\left\n\h_{a,\gamma}(f_1,f_2,f_3)\right\n_{\widehat{L}^p}&
=\left\n\F\left(K*(f_1\overline{f_2})\right)*\F f_3\right\n_{L^{p'}}\leq\left\n\F\left(K*(f_1\overline{f_2})\right)\right\n_{L^1}\left\n\F f_3\right\n_{L^{p'}}\\
&\asymp\left\n\F K\F(f_1\overline{f_2})\right\n_{L^1}\left\n f_3\right\n_{\widehat{L}^p}\leq\left\n\F K\right\n_{L^{\frac{p}{2}}}\left\n\F(f_1\overline{f_2})\right\n_{L^{\frac{p'}{2-p'}}}\left\n f_3\right\n_{\widehat{L}^p}\\
&\asymp\left\n\F K\right\n_{L^{\frac{p}{2}}}\left\n\F f_1*\F\overline{f_2}\right\n_{L^{\frac{p'}{2-p'}}}\left\n f_3\right\n_{\widehat{L}^p}\leq\left\n\F K\right\n_{L^{\frac{p}{2}}}\left\n\F f_1\right\n_{L^{p'}}\left\n\F\overline{f_2}\right\n_{L^{p'}}\left\n f_3\right\n_{\widehat{L}^p}.
\end{align*}
But $\left\n\F\overline{f_2}\right\n_{L^{p'}}=\left\n R\overline{\F f_2}\right\n_{L^{p'}}=\left\n\overline{\F f_2}\right\n_{L^{p'}}=\left\n\F f_2\right\n_{L^{p'}}$. This completes the proof.
\end{proof}

\section{Proofs of the main results}\label{pmt}


\begin{Remark}[Strategy of proof for local  well-posedness]\label{mip} It is known that $U(t): L^p(\R^d) \to \ L^p(\R^d)$ if and only if $p=2$.  For this reason, it is believed that, one cannot expect to solve NLS  with initial data in  $L^p(\R^d)$ $(p\neq2)$  as the linear counterpart of NLS is ill-posed in $L^p(\R^d).$ However, we can  overcome this difficulty via the following strategy:
\begin{enumerate}
\item[(i)] Apply  $U(-t)$  to the integral form of  $(\#)$, that is to \eqref{5},  and search for solution $\psi$ so that 
$$\phi(t)=U(-t)\psi(t)\in  X_T=\big(C([0,T],L^p(\rd)\cap L^2(\rd))\big)^N.$$
 Now notice that the linear counterpart of \eqref{t5} is well-posed in $L^p(\R^d)$.  This idea is inspired by the work of  Zhou \cite{zhou2010cauchy} 
for the NLS in $L^p (\R) \ (1<p<2).$  
\item[(ii)] Invoke factorization formula (Lemma \ref{2}) to obtain   transformed integral operator, say $\Phi$ (see \eqref{6}).
\item[(iii)] Choose closed ball of radius  $b,$  and centered at the origin, say $ \mathcal{V}_b^T$,  in $X_T$ (we note that the choice of $\mathcal{V}_b^T$ varies as  the Lebesgue space  exponent $p$ varies).
\item[(iv)]  Apply  trilinear (Subsection \ref{TE}) and   Strichartz estimates to obtain $\Phi:\mathcal{V}_b^T\to \mathcal{V}_b^T$ is contraction, and hence the local existence. 
 \end{enumerate}  
\end{Remark}

\begin{Remark}
We shall give the proof only for the Hartree-Fock equation \eqref{hf}. The proof for the reduced  Hartree-Fock equation \eqref{rhf}  can be proved similarly  and hence we  shall omit the details.
\end{Remark}
In this section we shall prove  our existence  theorems (Theorem  \ref{lw} to Theorem \ref{iglobal}). To  this end, we  start with the following technical lemma.

\begin{Lemma}\label{13}
Let $ t \in \R, u_1,u_2 \in \mathcal{S}(\R^d\times  \R),$ and 
 \begin{align}\label{nl1}
(\Omega (u_1,u_2))(t)&=\left(M_tU(-t)u_1(t)\right)*\left(R\overline{M_tU(-t)u_2(t)}\right).
\end{align}
Then for $0<\rho<\infty$ we have $
\n\F\left((\Omega (u_1,u_2))(t)\right)\n_{L^\rho}\lesssim|t|^{d(1-1/\rho)}\n u_1(t)\n_{L^{2\rho}}\n u_2(t)\n_{L^{2\rho}}.$
\end{Lemma}
\begin{proof}[{\bf Proof}]
Note that $\F R\overline{\varphi}=\overline{\F\varphi}$. Then
$
\F\left((\Omega (u_1,u_2))(t)\right)
=\left(\F M_tU(-t)u_1(t)\right)\left(\overline{\F M_tU(-t)u_2(t)}\right).
$
Therefore by H\"older inequality, we have 
\begin{equation}\label{14}
\n\F\left((\Omega (u_1,u_2))(t)\right)\n_{L^\rho}\leq\n\F M_tU(-t)u_1(t)\n_{L^{2\rho}}\n\F M_tU(-t)u_2(t)\n_{L^{2\rho}}.
\end{equation}
By  Lemma \ref{fl}, we have 
\begin{align*}
\n\F M_tU(-t)u_1(t)\n_{L^{2\rho}}&=\left(\int_{\rd}|D_t^{-1}M_t^{-1}u_1(t)|^{2\rho}  \right)^{1/2\rho}=\left((4\pi |t|)^{d\rho}\int_{\rd}|u_1(t)(4\pi tx)|^{2\rho}dx\right)^{1/2\rho}\\
&=\left((4\pi |t|)^{d\rho-d}\int_{\rd}|u_1(t)(x)|^{2\rho}dx\right)^{1/2\rho}=(4\pi |t|)^{(1-1/\rho)d/2}\n u_1(t)\n_{L^{2\rho}}
\end{align*}(here by $u_1(t)(y)$ we mean $u_1(t,y)$).
Using this  in \eqref{14},   we obtain the desired  inequality.
\end{proof}

\subsection{Local well-posedness in $L^p\cap L^2$}\label{plw}
%

\begin{proof}[{\bf First proof of Theorem \ref{lw}}]
By Duhamel's formula, we rewrite  \eqref{hf}  as
\begin{equation}\label{5}
\psi_k(t)=U(t)\psi_{0,k}+i\int_0^t U(t-s)(H_\psi\psi_k)(s)ds-i\int_0^t U(t-s)(F_\psi\psi_k)(s)ds.
\end{equation}
Writing $\psi_k(t)=U(t)\phi_k(t),$ we  have 
\begin{align}\label{t5}
\phi_k(t)&=\psi_{0,k}+i\int_0^t U(-s)(H_\psi\psi_k)(s)ds-i\int_0^t U(-s)(F_\psi\psi_k)(s)ds.
\end{align}
Let $\psi_0=(\psi_{0,1},\psi_{0,2},\cdots,\psi_{0,N})$. Using Lemma \ref{2}, we have 
\begin{align}\label{6}
\phi_k(t)&=\psi_{0,k}+ci\sum_{l=1}^N\sum_{j=1}^2\int_0^t s^{-\gamma}M_s^{-1}\widehat{\h}_{a,\gamma,s}^j(M_s\phi_l(s),RM_s\phi_l(s),M_s\phi_k(s))ds\\
&\ \ \ -ci\sum_{l=1}^N\sum_{j=1}^2\int_0^t s^{-\gamma}M_s^{-1}\widehat{\h}_{a,\gamma,s}^j(M_s\phi_k(s),RM_s\phi_l(s),M_s\phi_l(s))ds\nonumber:=\Phi_{\psi_0,k}(\mathbf{\phi})(t).
\end{align}
$\bullet$ \textbf{Case I:} $0<\gamma<\min\{1,\frac{d}{2}\}$ ($1\leq p\leq\infty$).

Let $q_1=\frac{8}{\gamma}, r=\frac{4d}{2d-\gamma},$
and introduce the space
\begin{align*}
V_b^T=\big\{v\in L_T^\infty(L^p(\rd)\cap L^2(\rd))&: \ \n v\n_{L_T^\infty(L^p\cap L^2)}\leq b,\\
&\ \n U(t)v(t)\n_{L_T^{q_2}(L^{2\rho})}\leq b,\ \n U(t)v(t)\n_{L_T^{q_1}(L^r)}\leq b\big\},
\end{align*} where $q_2,\rho$ to be chosen later.
We set
$\mathcal{V}_b^T=(V_b^T)^N$ and define the distance on it by
\begin{align*}
d(u,v)=\max\big\{\n u_j-v_j\n_{L_T^\infty(L^2\cap L^p)}&,\n U(t)[u_j(t)-v_j(t))]\n_{L_T^{q_2}(L^{2\rho})},\\
&\n U(t)[(u_j(t)-v_j(t)] \n_{L_T^{q_1}(L^r)}:\ j=1,2,\cdots,N\big\}.
\end{align*} Then  $(\mathcal{V}_b^T, d)$ is a complete metric space.
Next, we show that $\Phi_{\psi_0}:=(\Phi_{\psi_0,1},\cdots,\Phi_{\psi_0,N}),$ defined by  \eqref{6}, takes $\mathcal{V}_b^T$ into itself for suitable choices of $b$ and $T>0$.  Let $\phi=(\phi_1,...,\phi_N)\in \mathcal{V}_b^T.$ Denote
\begin{align}\label{inl}
I_{k,l,m}^j(t)=\int_0^t s^{-\gamma}M_s^{-1}\widehat{\h}_{a,\gamma,s}^j(M_s\phi_k(s),RM_s\phi_l(s),M_s\phi_m(s))ds,\quad j=1,2
\end{align} and 
\begin{eqnarray}\label{nl}
I_{k,l,m}(t)=I_{k,l,m}^1(t)+I_{k,l,m}^2(t).
\end{eqnarray}
By Remark \ref{le'}, we have 
\begin{align}\label{c1}
\left\n I_{k,l,m}^1(t)\right\n_{L^p}&=\left\n\int_0^t s^{-\gamma}M_s^{-1}\widehat{\h}_{a,\gamma,s}^1(M_s\phi_k(s),RM_s\phi_l(s),M_s\phi_m(s))ds\right\n_{L^p}\nonumber\\
&\lesssim\int_0^t s^{-\gamma}\n\phi_k(s)\n_{L^2}\n\phi_l(s)\n_{L^2}\n\phi_m(s)\n_{L^p} ds \lesssim b^3T^{1-\gamma}.
\end{align}
In view of \eqref{dtf} and \eqref{nl1}, we note that   
$$\widehat{\h}_{a,\gamma}^2(M_s\phi_l(s),RM_s\phi_l(s),M_s\phi_k(s))=\left[(k_2*S_{a,s})(\Omega(\psi_l, \psi_l))(s)\right]*(M_s\phi_k(s)).$$
By Young and H\"older  inequalities,  we  have 
\begin{align*}
\left\n I_{k,l,m}^2(t)\right\n_{L^p}&\lesssim\int_0^ts^{-\gamma}\n(k_2*S_{a,s})(\Omega(\psi_k,\psi_l))(s)\n_{L^1}\n\phi_m(s)\n_{L^p}ds\\
&\lesssim\int_0^ts^{-\gamma}\n(k_2*S_{a,s})\n_{L^\rho}\n\Omega(\psi_k,\psi_l)(s)\n_{L^{\rho'}}\n\phi_m(s)\n_{L^p}ds.
\end{align*}
Here we choose $\rho$ such that
\begin{eqnarray}\label{gr1}
\frac{d}{d-\gamma}<\rho<2.
\end{eqnarray} 
Note that we are able to choose such $\rho$ as $\gamma<d/2$.  By Lemmas \ref{3} and  \ref{13} and Hausdorff-Young inequality, we have
\begin{align*}
\left\n I_{k,l,m}^2(t)\right\n_{L^p}
&\lesssim\int_0^ts^{d-\gamma-d/\rho}\n\psi_k(s)\n_{L^{2\rho}}\n\psi_l(s)\n_{L^{2\rho}}\n\phi_m(s)\n_{L^p}ds.
\end{align*}
Note that $d-\gamma-\frac{d}{\rho}>0.$
Choose $q_2,q_3$ so that
\begin{eqnarray}\label{gr2}
\frac{1}{q_2}=\frac{d}{4}\bigg(1-\frac{1}{\rho}\bigg) \quad \text{and}\quad \frac{1}{q_3}=1-\frac{d}{2}\bigg(1-\frac{1}{\rho}\bigg).
\end{eqnarray} 
 By H\"older inequality, we have 
\begin{align}\label{c2}
\left\n I_{k,l,m}^2(t)\right\n_{L^p}&\lesssim T^{d-\gamma-d/\rho}\int_0^t\n\psi_k(s)\n_{L^{2\rho}}\n\psi_l(s)\n_{L^{2\rho}}\n\phi_m(s)\n_{L^p}ds\nonumber \\
&\lesssim T^{d-\gamma-d/\rho}\n\psi_k\n_{L_T^{q_2}(L^{2\rho})}\n\psi_l\n_{L_T^{q_2}(L^{2\rho})}\n\phi_m\n_{L_t^{q_3}(L^p)} \\
&\lesssim T^{d-\gamma-\frac{d}{\rho}+\frac{1}{q_3}}\n\psi_k\n_{L_T^{q_2}(L^{2\rho})}\n\psi_l\n_{L_T^{q_2}(L^{2\rho})}\n\phi_m\n_{L_t^\infty(L^p)}\nonumber\\
&\lesssim T^{d-\gamma-\frac{d}{\rho}+\frac{1}{q_3}}\n U(t)\phi_k\n_{L_T^{q_2}(L^{2\rho})}\n U(t)\phi_l\n_{L_T^{q_2}(L^{2\rho})}\n\phi_m\n_{L_t^\infty(L^p)} \lesssim T^{d-\gamma-\frac{d}{\rho}+\frac{1}{q_3}} b^3.\nonumber
\end{align}
Combining \eqref{6}, \eqref{c1} and  the above inequality, we have
\begin{align}\label{a1}
\left\n\Phi_{\psi_0,k}(\mathbf{\phi})\right\n_{L_T^\infty(L^p)}&\lesssim \n\psi_{0,k}\n_{L^p}+Nb^3(T^{1-\gamma}+T^{d-\gamma-\frac{d}{\rho}+\frac{1}{q_3}}).
\end{align}
For $(\underline{q},\underline{r})\in\{(q_1,r),(q_2,2\rho),(\infty,2)\}$ and $K=\frac{e^{-a|\cdot|}}{|\cdot|^\gamma},$   by  Proposition \ref{fst} we have 
\begin{align*}
\n U(t)I_{k,l,m}\n_{L^{\underline{q}}(L^{\underline{r}})}&\lesssim \n \left(K*(\psi_k\overline{\psi_l})\right)\psi_m\n_{L_T^{q_1'}(L^{r'})}.
\end{align*}
Note that 
$\frac{1}{{q_1}'}=\frac{4-\gamma}{4}+\frac{1}{q_1},\ \frac{1}{r'}=\frac{\gamma}{2d}+\frac{1}{r},\ \frac{4-\gamma}{4}=\frac{2}{q_1}+\frac{2-\gamma}{2}.$
By H\"older and Hardy-Littlewood-Sobolev inequalities, we have 
\begin{align}\label{b1}
\left\n \left(K*(\psi_k\overline{\psi_l})\right)\psi_m\right\n_{L_T^{{q_1}'}(L^{r'})}
&\leq\left\n\left\n K*(\psi_k\overline{\psi_l})\right\n_{L^{\frac{2d}{\gamma}}}\right\n_{L_T^{\frac{4}{4-\gamma}}}\left\n\left\n\psi_m\right\n_{L^r}\right\n_{L_T^{q_1}}\nonumber\\
&\leq\left\n\left\n|\cdot|^{-\gamma}*|\psi_k\overline{\psi_l}|\right\n_{L^{\frac{2d}{\gamma}}}\right\n_{L_T^{\frac{4}{4-\gamma}}}\left\n\psi_m\right\n_{L_T^{q_1}(L^r)}\nonumber\\
&\lesssim\left\n\left\n|\psi_k\overline{\psi_l}|\right\n_{L^{\frac{2d}{2d-\gamma}}}\right\n_{L_T^{\frac{4}{4-\gamma}}}\left\n\psi_m\right\n_{L_T^{q_1}(L^r)}\nonumber\\
&\leq\left\n\left\n\psi_k\right\n_{L^r}\left\n\psi_k\right\n_{L^r}\right\n_{L_T^{\frac{4}{4-\gamma}}}\left\n\psi_m\right\n_{L_T^{q_1}(L^r)}\nonumber\\
&\leq T^{1-\frac{\gamma}{2}}\left\n\psi_k\right\n_{L_T^{q_1}(L^r)}\left\n\psi_l\right\n_{L_T^{q_1}(L^r)}\left\n\psi_m\right\n_{L_T^{q_1}(L^r)}
\end{align}
and hence 
\begin{align}\label{a2}
\n U(t)I_{k,l,m}(t)\n_{L^{\underline{q}}(L^{\underline{r}})}&\lesssim T^{1-\gamma/2}\left\n\psi_k\right\n_{L_T^{q_1}(L^r)}\left\n\psi_l\right\n_{L_T^{q_1}(L^r)}\left\n\psi_m\right\n_{L_T^{q_1}(L^r)}.
\end{align}
Therefore by  \eqref{6} and Proposition \ref{fst}, we have 
\begin{align}\label{a2'}
\n U(t)\Phi_{\psi_0,k}(\mathbf{\phi})\n_{L^{\underline{q}}(L^{\underline{r}})}&\lesssim \n\psi_{0,k}\n_{L^2}+Nb^2T^{1-\gamma/2}.
\end{align}
Choose $b=2c\n\psi_0\n_{(L^2\cap L^p)^N}$ and $T>0$ small enough so that \eqref{a1}, \eqref{a2'} imply $\Phi_{\psi_0} (\phi)\in \mathcal{V}_b^T.$\\
Note that by tri-linearity of $\widehat{\h}_{a,\gamma, t}$, we have
\begin{align}\label{7}
\widehat{\h}_{a,\gamma,s}(f_1,f_2,f_3)-\widehat{\h}_{a,\gamma,s}(g_1,g_2,g_3)&=\widehat{\h}_{a,\gamma,s}(f_1-g_1,f_2,f_3)+\widehat{\h}_{a,\gamma,s}(g_1,f_2-g_2,f_3)\\
&\ \ \ +\widehat{\h}_{a,\gamma,s}(g_1,g_2,f_3-g_3).\nonumber
\end{align}
Using this, for $u,v\in\mathcal{V}_b^T,$  
 arguing as above in \eqref{a1}, we have
\begin{align}\label{a3}
\left\n\Phi_{\psi_0,k}(u)-\Phi_{\psi_0,k}(v)\right\n_{L_T^\infty(L^p)}&\lesssim Nb^2(T^{1-\gamma}+T^{d-\gamma-\frac{d}{\rho}+\frac{1}{q_3}})d(u,v)
\end{align}
and on the other hand as in \eqref{a2}
\begin{align}\label{a4}
\n U(t)[\Phi_{\psi_0,k}(u)(t)-\Phi_{\psi_0,k}(v)(t)]\n_{L^{\underline{q}}(L^{\underline{r}})}\lesssim T^{1-\gamma/2}Nb^2d(u,v).
\end{align}
Using  \eqref{a3} and \eqref{a4}, we may conclude that  $\Phi_{\psi_0}: \mathcal{V}_b^T\to \mathcal{V}_b^T$ is  a contraction provided $T$ is sufficiently small (depending on  $\|\psi_{0,1}\|_{L^p\cap L^2},..., \|\psi_{0, N}\|_{L^p\cap L^2}, d, \gamma, N$). Then, by Banach contraction principle, there exists a unique $(\phi_1, ...,\phi_N) \in \mathcal{V}_b^T$ solving \eqref{6}.\\
$\bullet$ \textbf{Case II:} $0<\gamma<\min\{1,2d(\frac{1}{p}-\frac{1}{2})\}$ and $1\leq p<2$ (improves Case I when $1\leq p<\frac{4}{3}$).

For $b, T>0,$ let
$$V_b^T=\{v\in L^\infty\big((0,T),L^p\cap L^2(\rd)\big):\n v\n_{L^\infty_T(L^p\cap L^2)}\leq b\}.$$  We set $\mathcal{V}_b^T=(V_b^T)^N$ and define the distance on it by
$$d(u,v)=\max\big\{\n u_j-v_j\n_{L^\infty_T(L^p\cap L^2)}:j=1,2,\cdots,N\big\},$$
where $u=(u_1,u_2\cdots,u_N),v=(v_1,v_2,\cdots,v_N)\in\mathcal{V}_b^T$. Then  $(\mathcal{V}_b^T, d)$ is a complete metric space.
Next, we show that the mapping $\Phi_{\psi_0},$ defined by  \eqref{6}, takes $\mathcal{V}_b^T$ into itself for  a suitable choice of $b$ and small  $T>0$.  Let $\phi=(\phi_1,...,\phi_N)\in \mathcal{V}_b^T.$ Choose $b$ so that $\n\psi_0\n_{(L^p\cap L^2)^N}=b/2$. 
Then, by Proposition \ref{le}\eqref{lea}, for $0<t<T$, we obtain
\begin{align*}
\n\Phi_{\psi_0,k}(\mathbf{\phi})(t)\n_{L^p\cap L^2}&\leq\frac{b}{2}+cN\int_0^t s^{-\gamma}\n M_s\phi_l(s)\n_{L^p\cap L^2}^2\n M_s\phi_k(s)\n_{L^p\cap L^2}ds\\
&\leq \frac{b}{2}+2cNb^3\int_0^ts^{-\gamma}ds=\frac{b}{2}+\frac{2cNb^3}{1-\gamma}T^{1-\gamma}
\end{align*}
 Now we choose $T>0$ small enough so that $\frac{2cNb^2}{1-\gamma}T^{1-\gamma}\leq\frac{1}{2}$ to achieve $\n\Phi_{\psi_0,k}(\mathbf{\phi})(t)\n_{L^p\cap L^2}\leq b$
 for all  $k=1,..., N.$  Hence $\Phi_{\psi_0}$ is a map from $\mathcal{V}_b^T$ to itself with the above choices of $b$ and $T$.
For $u,v\in\mathcal{V}_b^T,$ by Proposition \ref{le} \eqref{lea}, \eqref{7} and arguing as above, we obtain
\begin{align*}
d\left(\Phi_{\psi_0}(u),\Phi_{\psi_0}(v)\right)\lesssim Nb^2\int_0^t s^{-\gamma}\n u_l(s)-v_l(s)\n_{L^p\cap L^2}ds\lesssim Nb^2T^{1-\gamma}d\left(u,v\right).
\end{align*}
Thus   $\Phi_{\psi_0}: \mathcal{V}_b^T\to \mathcal{V}_b^T$ is a contraction map provided that $T$ is sufficiently small (depending on  $\|\psi_{0,1}\|_{L^p\cap L^2},..., \|\psi_{0, N}\|_{L^p\cap L^2}, d, \gamma,N$). Then, by Banach contraction principle, there exists a unique $(\phi_1, ...,\phi_N) \in \mathcal{V}_b^T$ solving \eqref{6}.
\end{proof}
 In \cite{zhou2010cauchy},  Zhou proved local existence  for the  cubic NLS in  $L^p (\R)$  by 
introducing a function space (to be defined below) based on the fundamental theorem
of calculus and the Schr\"odinger propogator.   Specifically, for $T>0,1\leq p,q\leq\infty$ and $\theta\geq0,$ the  Zhou spaces  $\widetilde{X}_{q,\theta}^p(T),$   $\widetilde{Y}_{q,\theta}^p(T),$ and $Y_{q,\theta}^p(T)$  are given  by 
$$\widetilde{X}_{q,\theta}^p(T)=\big\{v:[0,T]\times\rd\rightarrow\C\ :\ \n v\n_{\widetilde{X}_{q,\theta}^p(T)}:=\n t^{\theta}\n(\partial_tv)(t,\cdot)\n_{L^p}\n_{L_T^q}<\infty\big\},$$
$$\widetilde{Y}_{q,\theta}^p(T)=\big\{v\in\widetilde{X}_{q,\theta}^p(T)\ :\n v\n_{\widetilde{Y}_{q,\theta}^p(T)}:=\n v(0)\n_{L^p}+\n v\n_{\widetilde{X}_{q,\theta}^p(T)}< \infty  \big\}$$
and
\begin{equation}\label{zs}
Y_{q,\theta}^p(T)=\big\{u:[0,T]\times\rd\rightarrow\C\ |\ U(-t)u(t)\in\widetilde{Y}_{q,\theta}^p(T)\big\}.
\end{equation}

Now using the Zhou spaces, we will  briefly give a different proof of local  existence  in $L^p(\R^d)\cap L^2(\R^d)$  $(p\in[1,\infty]\setminus\{2\})$  for $(\#)$ with certain $\gamma$'s. 
Strictly speaking,  in this case,  we  will prove the local existence
in $\widetilde{Y}_{\infty,\gamma}^p(T)$  which is continuously embedded in  $C([0, T], L^p(\R^d))$: 
\begin{Lemma}[Lemma 2.1 \cite{hyakuna2019global}]\label{inclu1}
Let $0<T<\infty, p\geq1$ and $0<\gamma<1.$ Then
$\widetilde{Y}_{\infty,\gamma}^p(T)\hookrightarrow C([0, T], L^p(\R^d)).$
\end{Lemma}
The key estimates for $\widetilde{Y}_{\infty, \gamma}^{p}(T)$-regularity of the local existence is:

\begin{Lemma}\label{8} Denote the Duhamel type operator by
\begin{equation}\label{dhtp}
\mathcal{D}_{a,\gamma}(v_1,v_2,v_3)(t)=\int_0^tM_s^{-1}s^{-\gamma}\widehat{\h}_{a,\gamma,s}(M_sv_1(s),RM_sv_2(s),M_sv_3(s))ds.
\end{equation}
\begin{enumerate}
\item\label{8a} Assume $0<\gamma<1$. Then
$\n \mathcal{D}_{a,\gamma}(v_1,v_2,v_3)\n_{\widetilde{X}_{2/\gamma,0}^2(T)}\lesssim\prod_{l=1}^3\n v_l\n_{\widetilde{Y}_{1,0}^2(T)}.$
\item\label{8b} Assume $0<\gamma<\min\{1,2d(\frac{1}{p}-\frac{1}{2})\}$ when $1\leq p<2$ and $0<\gamma<\min\{1,d(\frac{1}{2}-\frac{1}{p})\}$ when $2<p\leq\infty$. Then
$\n \mathcal{D}_{a,\gamma}(v_1,v_2,v_3)\n_{\widetilde{X}_{\infty,\gamma}^p(T)}\lesssim\prod_{j=1}^3\n v_j\n_{\widetilde{Y}_{1,0}^2(T)\cap\widetilde{Y}_{1,0}^p(T)}.$
\end{enumerate}
\end{Lemma}
\begin{proof}[{\bf Proof}]
\eqref{8a} In view of \eqref{tf}, we have 
\begin{align*}
\widehat{\h}_{a,\gamma,t}(M_tv_1(t),RM_tv_2(t),M_tv_3(t))&=\left[\left(S_{a,t}*|\cdot|^{\gamma-d}\right)(M_tv_1(t)*\overline{RM_tv_2(t)})\right]*M_tv_3(t).
\end{align*}
Then using Lemma \ref{2} we get
\begin{align*}
I:=\left\n\partial_tD_{a,\gamma}(v_1,v_2,v_3)\right\n_{L_T^{2/\gamma}(L^2)}&\asymp\left\n t^{-\gamma}\widehat{\h}_{a,\gamma,t}(M_tv_1(t),RM_tv_2(t),M_tv_3(t))\right\n_{L_T^{2/\gamma}(L^2)}\\
&\asymp\left\n U(-t)\h_{a,\gamma}(U(t)v_1(t),U(t)v_2(t),U(t)v_3(t)\right\n_{L_T^{2/\gamma}(L^2)}\\
&=\left\n \h_{a,\gamma}(U(t)v_1(t),U(t)v_2(t),U(t)v_3(t)\right\n_{L_T^{2/\gamma}(L^2)}.
\end{align*}
By  H\"older and  Hardy-Littlewood-Sobolev inequalities, we have 
\begin{align*}
I&\lesssim\left\n\left\n\left(|\cdot|^{-\gamma}e^{-a|\cdot|}\right)*\left((U(t)v_1(t))\overline{(U(t)v_2(t))}\right)\right\n_{L^{\frac{3d}{\gamma}}}\n U(t)v_3(t)\n_{L^{\frac{6d}{3d-2\gamma}}}\right\n_{L_T^{\frac{2}{\gamma}}}\\
&\lesssim\left\n\left\n|\cdot|^{-\gamma}*\left|(U(t)v_1(t))\overline{(U(t)v_2(t))}\right|\right\n_{L^{\frac{3d}{\gamma}}}\n U(t)v_3(t)\n_{L^{\frac{6d}{3d-2\gamma}}}\right\n_{L_T^{\frac{2}{\gamma}}}\\
&\lesssim\left\n\left\n(U(t)v_1(t))\overline{(U(t)v_2(t))}\right\n_{L^{\frac{3d}{3d-2\gamma}}}\n U(t)v_3(t)\n_{L^{\frac{6d}{3d-2\gamma}}}\right\n_{L_T^{\frac{2}{\gamma}}}\\
&\leq\left\n\prod_{l=1}^3\n U(t)v_l(t)\n_{L^{\frac{6d}{3d-2\gamma}}}\right\n_{L_T^{\frac{2}{\gamma}}}\leq\prod_{l=1}^3\n U(t)v_l(t)\n_{L_T^{\frac{6}{\gamma}}\big(L^{\frac{6d}{3d-2\gamma}}\big)}.
\end{align*}
Using the fundamental theorem of calculus, we have \begin{align}\label{9}
v_l(t)=v_l(0)+\int_0^t\partial_sv_l(s)ds\quad\Longrightarrow\quad  U(t)v_l(t)=U(t)v_l(0)+\int_0^tU(t)\partial_sv_l(s)ds.
\end{align}
Put $q=\frac{2}{\gamma}$ and $r=\frac{6d}{3d-2\gamma}$. 
Then  by the above equality we have
\begin{align*}
\n U(t)v_l(t)\n_{L_T^{3q}(L^r)}&\leq\n U(t)v_l(0)\n_{L_T^{3q}(L^r)}+\left\n\int_0^tU(t)\partial_sv_l(s)ds\right\n_{L_T^{3q}(L^r)}\\
&\leq\n U(t)v_l(0)\n_{L_T^{3q}(L^r)}+\left\n\int_0^t\n U(t)\partial_sv_l(s)\n_{L^r}ds\right\n_{L_T^{3q}}.
\end{align*}
By Minkowski inequality, we have 
\begin{align*}
\left\n\int_0^t\n U(t)\partial_sv_l(s)\n_{L^r}ds\right\n_{L_T^{3q}}
&\leq\int_0^T\left(\int_0^T\n U(t)\partial_sv_l(s)\n_{L^r}^{3q}dt\right)^{1/3q}ds=\int_0^T\n U(t)\partial_sv_l(s)\n_{L^{3q}(L^r)}ds.
\end{align*}
Therefore, $(3q,r)$ being a 2-admisible pair, we get 
\begin{equation*}
\n U(t)v_l(t)\n_{L_T^{3q}(L^r)}\lesssim\n v_l(0)\n_{L^2}+\int_0^T\n \partial_sv_l(s)\n_{L^2}ds=\n v_l\n_{\w{Y}_{1,0}^2(T)}.
\end{equation*}
\eqref{8b} In view of  Proposition \ref{le},  we obtain
\begin{align*}
\n \mathcal{D}_{a,\gamma}(v_1,v_2,v_3)\n_{\widetilde{X}_{\infty,\gamma}^p(T)}&=\sup_{s\in[0,T]}s^\gamma\n\partial_s\mathcal{D}_{a,\gamma}(v_1,v_2,v_3)(s)\n_{L^p}\\
&=\sup_{s\in[0,T]}\n\widehat{\h}_{a,\gamma,s}(M_sv_1(s),RM_sv_2(s),M_sv_3(s))\n_{L^p}\\
&\lesssim\sup_{s\in[0,T]}\prod_{l=1}^3\n v_l(s)\n_{L^2\cap L^p}\lesssim\prod_{l=1}^3\n v_l\n_{\widetilde{Y}_{1,0}^2(T)\cap\widetilde{Y}_{1,0}^p(T)},
\end{align*}
where the last inequality follows from \eqref{9} by taking $L^2,L^p$-norms on both sides of it.
\end{proof}

\begin{Remark} \label{SR1}
The local well-posedenss proved in \cite[Theorem 1.3]{hyakuna2018global}  for Hartree equation in $L^p(\rd)\cap L^2(\rd)$  for $1<p\leq 2$  depend on  \cite[Proposition 4.1]{hyakuna2018global}. We found that there is a subtle flaw in it. 
Specifically, while proving a result similar to Lemma \ref{8} \eqref{8a} with $a=0,1<p\leq2$ in \cite[see eq. (4.6), p.1093]{hyakuna2018global},  the following  inequality 
\begin{equation}\label{100}
\n t^{-\gamma}\widehat{\mathcal{H}}_{0,\gamma,t}^j(v_1,v_2,v_3)\n_{L_T^q(L^2)}\leq c\n t^{-\gamma}\widehat{\mathcal{H}}_{0,\gamma,t}(v_1,v_2,v_3)\n_{L_T^q(L^2)}
\end{equation}
where  $\widehat{\mathcal{H}}_{0,\gamma,t}$ and $ \widehat{\mathcal{H}}_{0,\gamma,t}^j$ are as in \eqref{tf} and \eqref{dtf} respectively, has been  crucially used. However, we find that    \eqref{100} is not true. 

We shall justify our claim by producing a counter example. Note that  \eqref{100} implies 
 \begin{equation}\label{1001}
\n \widehat{\mathcal{H}}_{0,\gamma,t}^j(v_1,v_2,v_3)\|_{L^2}\leq c\n \widehat{\mathcal{H}}_{0,\gamma,t}(v_1,v_2,v_3)\n_{L^2}
\end{equation} 
 for all $t$ in  some positive measured subset of $[0, T]$.  If possible, we assume that $\eqref{100}$ is true. Let $k_1$  and $k_2$ are as  in \eqref{di}  and $k:=k_1+k_2.$ Then, by Plancherel theorem, we have 
   \begin{align*}
\n\widehat{\mathcal{H}}_{0,\gamma,t}^j(f,g,h)\n_{L^2}\leq c\n\widehat{\mathcal{H}}_{0,\gamma,t}(f,g,h)\n_{L^2}&\Longleftrightarrow\n [k_j(f\ast\bar{g})]\ast h\n_{L^2}\leq c\n [k(f\ast\bar{g})]\ast h\n_{L^2}\\
&\Longleftrightarrow\n \F [k_j(f\ast\bar{g})]\F{h} \n_{L^2}\leq c\n \F[k(f\ast\bar{g})]\F h\n_{L^2}.
\end{align*}
Thus, we have
\begin{align*}
\int_{\rd}\left(\left|\F [k_j(f\ast\bar{g})]\right|^2-c^2\left|\F [k(f\ast\bar{g})]\right|^2\right)|\F h|^2\leq0.
\end{align*} Since $h$ is arbitrary we conclude $\left|\F [k_j(f\ast\bar{g})](\xi)\right|\leq c\left|\F [k(f\ast\bar{g})](\xi)\right|$ for a.e. $\xi\in\rd$ . If we choose compactly supported  $f\ast\bar{g}$ in the space $L^\infty(\rd)$, so that $\F [k_j(f\ast\bar{g})],\F [k(f\ast\bar{g})]$ are continuous, then 
we would have 
\begin{align}\label{110}
\left|\F [k_j(f\ast\bar{g})](\xi)\right|\leq c\left|\F [k(f\ast\bar{g})](\xi)\right|\quad\text{ for all }\xi\in\rd,\ j=1,2.
\end{align} 
\begin{center}
\begin{figure}
\begin{tikzpicture}[scale=.75]
\draw[dashed,gray] (-7,0)--(7,0) node[anchor=west] {\tiny{}};
\draw[dashed,gray][->] (0,-2)--(0,2) node[anchor=north east] {\tiny{$ $}};
\draw[<-] (-7,0)--(-5,0)--(-3.5,-1.5)--(-2,0)--(-.7,0)--(0,.7)--(.7,0)--(2,0);
\draw[->](2,0)--(3.5,-1.5)--(5,0)--(7,0);
\draw[dashed,gray] (-6,.7)--(6,.7)node[anchor=west]{\textcolor{black}{\tiny{$a$}}};
\draw[dashed,gray] (-6,-1.5)--(6,-1.5)node[anchor=west]{\textcolor{black}{\tiny{$-b$}}};
\filldraw[black](.7,0)circle(1pt)node[anchor=north ]{\tiny{$(a,0)$}};
\filldraw[black](-.7,0)circle(1pt)node[anchor=north ]{\tiny{$(-a,0)$}};
\filldraw[black](2,0)circle(1pt)node[anchor=south ]{\tiny{$(2,0)$}};
\filldraw[black](-2,0)circle(1pt)node[anchor=south ]{\tiny{$(-2,0)$}};
\filldraw[black](3.5,0)circle(1pt)node[anchor=north ]{\tiny{$(2+b,0)$}};
\filldraw[black](-3.5,0)circle(1pt)node[anchor=north ]{\tiny{$(-2-b,0)$}};
\filldraw[black](5,0)circle(1pt)node[anchor=south ]{\tiny{$(2+2b,0)$}};
\filldraw[black](-5,0)circle(1pt)node[anchor=south ]{\tiny{$(-2-2b,0)$}};
\end{tikzpicture}
\caption{{Graph of $H_{a,b}=f\ast\bar{g}$ is indicated in black line for $d=1$.}}\label{fig1}
\end{figure}
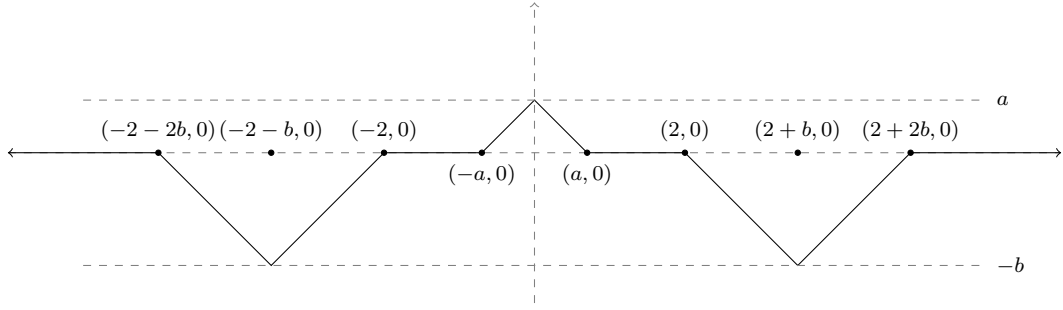
\end{center} 
For simplicity we  take dimension $d=1$ and so $0<\gamma<1.$ Let $0<a<1, b>0$ (to be chosen later). Define  $H_{a,b}:\mathbb R \to \mathbb R$ as follows
\begin{align*}
H_{a,b}(x)=
\begin{cases}
a-x, &\text{ for }0\leq x<a\\
0, &\text{ for }a\leq x<2\text{ and }2+2b\leq x<\infty\\
2-x, &\text{ for }2\leq x<2+b\\
x-2-2b, &\text{ for }2+b\leq x<2+2b\\
H_{a,b}(-x), &\text{ for }-\infty< x<0.
\end{cases}
\end{align*} The graph of  $H_{a,b}$  is given in Figure \ref{fig1} and it is clear that $H_{a,b}\in L^\infty(\R)$ is compactly supported. Note that $H_{a,b}$ is a finite linear combinations of triangle functions.  So $H_{a,b}$ can also be viewed as the finite linear combination of functions of the form $\chi_{A}\ast\chi_B$ where $A,B$ are finite intervals in $\R$.  Thus, we have $H_{a,b} \in \widehat{L}^{\infty}(\mathbb R)$. We recall Riesz factorization criterion:  $\widehat{L}^\infty(\rd)=L^2(\rd)\ast L^2(\rd)$, see for e.g., \cite[Theorem B]{liflyand2011conditions}. In view of this, there exists $f, g \in L^2(\mathbb R)$ such that $H_{a,b}=f\ast \bar{g}.$ We shall  now compute the LHS and RHS of \eqref{110} at point 0.
In fact, 
\begin{align*}
\F [k(f\ast\bar{g})](0)&=\int_\R|x|^{\gamma-1}(f\ast\bar{g})(x) dx\\
&=2\int_0^a x^{\gamma-1}(a-x)dx-2\int_2^{2+b}x^{\gamma-1}(x-2)dx-2\int_{2+b}^{2+2b}x^{\gamma-1}(2+2b-x)dx\\
&=\frac{2}{\gamma(\gamma+1)}\left[a^{\gamma+1}-\{(2+2b)^{\gamma+1}-2(2+b)^{\gamma+1}+2^{\gamma+1}\}\right]= \frac{2}{\gamma(\gamma+1)}\left[a^{\gamma+1}-G(b)\right],
\end{align*}
where $G(b)=(2+2b)^{\gamma+1}-2(2+b)^{\gamma+1}+2^{\gamma+1}.$ 
Note that $G(0)=0$ and by using strict convexity of $x\mapsto x^{1+\gamma}$ on $(0,\infty)$ for $\gamma>0$, we have  $G(b)>0$ for $b>0$. Choose  $0<a<1$ small and adjust $b>0$ so that $a^{\gamma+1}=G(b)$. Hence, we obtain  $\F [k(f\ast\bar{g})](0)=0.$ But note that 
\begin{align*}
\left|\F [k_1(f\ast\bar{g})](0)\right|=\int_{|x|\leq1}|x|^{\gamma-1}(f\ast\bar{g})(x)
=2\int_0^a x^{\gamma-1}(a-x)dx=\frac{2a^{\gamma+1}}{\gamma(\gamma+1)}>0
\end{align*} and this contradicts \eqref{110} as the RHS is zero in \eqref{110} for $\xi=0$. Thus \eqref{100} is not true.
\end{Remark}
\begin{Remark}\label{SR2} Because of the  issue discussed in Remark \ref{SR1} we do not decompose  Duhamel type operator  $\mathcal{D}_{a,\gamma}$ (defined in \eqref{dhtp} above) in two parts. We rather  work on $\mathcal{D}_{a,\gamma}$ without decomposing into two parts  but still get the nonlienar estimates of  Lemma \ref{8}.
\end{Remark}

\begin{proof}[{\bf Second proof of Theorem \ref{lw}}] Here we assume $0<\gamma<\min\{1,2d(\frac{1}{p}-\frac{1}{2})\}$ when $1\leq p<2$ and $0<\gamma<\min\{1,d(\frac{1}{2}-\frac{1}{p})\}$ when $2<p\leq\infty$.\\
For  $b, T>0,$ let 
$V_b^T(v_0)=\big\{v\in\widetilde{Y}_{2/\gamma,0}^2(T)\cap\widetilde{Y}_{\infty,\gamma}^p(T)\ : \ \n v\n_{\widetilde{X}_{2/\gamma,0}^2(T)\cap\widetilde{X}_{\infty,\gamma}^p(T)}\leq b,v(0)=v_0\big\}.$  For $\psi_0=(\psi_{0,1},\cdots,\psi_{0,N})\in (L^p(\rd)\cap L^2(\rd))^N$, introduce the space 
$\mathcal{V}_b^T(\psi_0)=V_b^T(\psi_{0,1})\times V_b^T(\psi_{0,2})\times\cdots\times V_b^T(\psi_{0,N})$
 and define the distance on it by 
$d(u,v)=\max\big\{\n u_j-v_j\n_{\widetilde{X}_{2/\gamma,0}^2(T)\cap\widetilde{X}_{\infty,\gamma}^p(T)}:\ j=1,2,\cdots,N\big\}.$
Next, we show that the mapping $\Phi_{\psi_0},$ defined by  \eqref{6}, takes $\mathcal{V}_b^T(\psi_0)$ into itself for suitable choice of $b$ and small  $T>0$.  Let $\phi=(\phi_1,...,\phi_N)\in \mathcal{V}_b^T(\psi_0).$ Since $\n\psi_{k,0}\n_{\widetilde{X}_{2/\gamma,0}^2(T)\cap\widetilde{X}_{\infty,\gamma}^p(T)}=0,$  by  Lemma  \ref{8}, we have 
\begin{align*}
&\n \Phi_{\psi_0,k}(\phi)\n_{\widetilde{X}_{2/\gamma,0}^2(T)\cap\widetilde{X}_{\infty,\gamma}^p(T)}\\
&\lesssim \sum_{l=1}^N\n \mathcal{D}_{a,\gamma}(\phi_l,\phi_l,\phi_k)\n_{\widetilde{X}_{2/\gamma,0}^2(T)\cap\widetilde{X}_{\infty,\gamma}^p(T)}+\sum_{l=1}^N\n \mathcal{D}_{a,\gamma}(\phi_k,\phi_l,\phi_l)\n_{\widetilde{X}_{2/\gamma,0}^2(T)\cap\widetilde{X}_{\infty,\gamma}^p(T)}\\
&\lesssim\n\phi_k\n_{\widetilde{Y}_{1,0}^2(T)\cap\widetilde{Y}_{1,0}^p(T)}\sum_{l=1}^N\n\phi_l\n_{\widetilde{Y}_{1,0}^2(T)\cap\widetilde{Y}_{1,0}^p(T)}^2.
\end{align*}
By H\"older inequality, we have $\n\phi_l\n_{\widetilde{X}_{1,0}^2(T)} \leq T^{1-\gamma/2}\n\phi_l\n_{\widetilde{X}_{\gamma/2,0}^2(T)}$ and $\n\phi_l\n_{\widetilde{X}_{1,0}^p(T)} \leq T^{1-\gamma}\n\phi_l\n_{\widetilde{X}_{\infty,\gamma}^p(T)}$.
Therefore, we  have 
$$\n\phi_l\n_{\widetilde{Y}_{1,0}^2}\lesssim\n\phi_{l,0}\n_{L^2}+T^{1-\gamma/2}\n\phi_l\n_{\widetilde{X}_{2/\gamma,0}^2(T)}\ \ \text{ and }\ \ \n\phi_l\n_{\widetilde{Y}_{1,0}^p}\lesssim\n\phi_{l,0}\n_{L^p}+T^{1-\gamma}\n\phi_l\n_{\widetilde{X}_{\infty,\gamma}^p(T)}.$$
Hence, by taking $0<T<1,$ we obtain
\begin{align*}
\n \Phi_{\psi_0,k}(\phi)\n_{\widetilde{X}_{2/\gamma,0}^2(T)\cap\widetilde{X}_{\infty,\gamma}^p(T)}\lesssim N\left(\n\psi_0\n_{L^2\cap L^p}^3+T^{3(1-\gamma)}b^3\right).
\end{align*}
We set $b=2cN\n\psi_0\n_{(L^p\cap L^2)^N}^3$
 and $T>0$ small so that  we have $\n \Phi_{\psi_0,k}(\phi)\n_{\widetilde{X}_{2/\gamma,0}^2(T)\cap\widetilde{X}_{\infty,\gamma}^p(T)}\leq b.$
Consequently, we have $\Phi_{\psi_0} (\phi)\in\mathcal{V}_b^T(\psi_0)$. To check $\Phi_{\psi_0}$ is a contraction, for $u,v\in\mathcal{V}_b^T(\psi_0)$,
using \eqref{7} together with Lemma \ref{8}, we can conclude
\begin{align*}
d\left(\Phi_{\psi_0}(u),\Phi_{\psi_0}(v)\right)&\lesssim T^{1-\gamma}(\n\psi_0\n_{L^2\cap L^p}^2+T^{2(1-\gamma)}b^2)\sum_{l=1}^N\n u_l-v_l\n_{\widetilde{X}_{2/\gamma,0}^2(T)\cap\widetilde{X}_{\infty,\gamma}^p(T)}\\
&\lesssim NT^{1-\gamma}(\n\psi_0\n_{L^2\cap L^p}^2+T^{2(1-\gamma)}b^2)d(u,v).
\end{align*}
Thus $\Phi_{\psi_0}:\mathcal{V}_b^T(\psi_0) \to \mathcal{V}_b^T(\psi_0)$ is a contraction provided $T>0$ is small enough.\\
 Finally, we note that 
 by Proposition \ref{miF}, we  have  $\psi\in C([0,T],L^2(\rd))^N,$ and consequently $\phi\in C([0,T],L^2(\rd))^N$. By Lemma \ref{inclu1} we have $\phi\in C([0,T],L^p(\rd))^N$ for $0<\gamma<1$.
\end{proof}

\subsection{Local well-posedness in $\widehat{L}^p\cap L^2$}\label{plhw}

%

\begin{Lemma}\label{4}
For all $t\in \mathbb R$ and $0<\alpha<\infty$  the fractional Schr\"odinger propagator $U_\alpha(t)=e^{-it(-\Delta)^{\alpha/2}}$  is an isometry on $\widehat{L}^p(\mathbb R^d) \ (1\leq p \leq \infty),$ that is, 
$\|e^{-it (-\Delta)^{\alpha/2}}f\|_{\widehat{L}^p}=\|f\|_{\widehat{L}^p}.$
\end{Lemma}

%
%

\begin{proof}[{\bf First proof of Theorem \ref{lhw}}]
By Duhamel's formula, we rewrite  \eqref{hf} as 
\begin{align}\label{12}
\psi_k(t)&=U_{\alpha}(t)\psi_{0,k}+i\int_0^t U_\alpha(t-s)(H_\psi\psi_k)(s)ds-i\int_0^t U_\alpha(t-s)(F_\psi\psi_k)(s)ds\\
& :=\Psi_{\psi_0, k}(\psi)(t).\nonumber
\end{align}
$\bullet$ \textbf{Case I:} $0<\gamma<\min\{\alpha,\frac{d}{2}\}$ ($1\leq p\leq\infty$). 

 Hereafter, for  $\alpha\in(\frac{2d}{2d-1},2),$ we assume  initial data are radial and $d\geq 2.$  In fact, in this case, the  members of  $U_b^T,$ to be defined below, are radial functions.   For the  notational convenience, we omit mentioning this   explicitly in the proof below.
Let $s=\alpha/2$ and ${q_1}=\frac{8s}{\gamma}, r=\frac{4d}{2d-\gamma},$ and  for $T, b>0,$  introduce the space
$$U_b^T=\big\{v\in L_T^\infty\big(L^2(\rd)\cap\widehat{L}^p(\rd)\big):\n v\n_{L_T^\infty(L^2\cap\widehat{L}^p)}\leq b,\n v\n_{L_T^{q_1}(L^r)}\leq b, \n v\n_{L_T^{2q_2}(L^{2\rho})}\leq b\big\},$$
where $q_2,\rho$ to be chosen later.   We set $\mathcal{U}_b^T=(U_b^T)^N$ and  the  distance on it by
$$d(u,v)=\max\big\{\n u_j-v_j\n_{L_T^\infty(L^2\cap\widehat{L}^p)},\n u_j-v_j\n_{L_T^{q_1}(L^r)},\n u_j-v_j\n_{L_T^{2q_2}(L^{2\rho})}:j=1,2,\cdots,N\big\},$$
where $u=(u_1,u_2\cdots,u_N),v=(v_1,v_2,\cdots,v_N)\in\mathcal{U}_b^T$.  Next, we show that the mapping $\Psi_{\psi_0}=(\Psi_{\psi_0,1},\cdots,\Psi_{\psi_0,N}),$ defined by  \eqref{12}, takes $\mathcal{U}_b^T$ into itself for suitable choice of $b$ and small  $T>0$.  Let $\psi=(\psi_1,...,\psi_N)\in \mathcal{U}_b^T.$ 
Denote 
\begin{align}\label{f1}
J_{k,l,m}(t)=\int_0^tU_\alpha(t-s)\h_{a,\gamma}(\psi_k(s),\psi_l(s),\psi_m(s))ds.
\end{align}
Let $\frac{d}{d-\gamma}<\rho\leq2,$  $h_a(\xi)= \frac{c_da}{(a^2+ 4\pi^2|\xi|^2)^{(d+1)/2}}$ (see  \eqref{di} and Lemma \ref{1}) and  $K=\frac{e^{-a|\cdot|}}{|\cdot|^\gamma}$. We have 
\begin{align*}
\left\n\h_{a,\gamma}(f,g,h)\right\n_{\widehat{L}^p}&=\left\n\F\left[\left(K*(f\overline{g})\right)h\right]\right\n_{L^{p'}}\leq\left\n\F\left[K*(f\overline{g})\right]\right\n_{L^1}\left\n\F h\right\n_{L^{p'}}
=\left\n\F K\F(f\overline{g})\right\n_{L^1}\left\n h\right\n_{\widehat{L}^p}\\
& \leq (\left\n (k_1\ast h_a)\right\n_{L^1}\left\n\F(f\overline{g})\right\n_{L^\infty}+\left\n (k_2*h_a)\right\n_{L^\rho}\left\n\F(f\overline{g})\right\n_{L^{\rho'}}) \left\n h\right\n_{\widehat{L}^p}\\
&\lesssim (\left\n f\overline{g}\right\n_{L^1}+\left\n f\overline{g}\right\n_{L^\rho})  \left\n h\right\n_{\widehat{L}^p}
\leq (\left\n f\right\n_{L^2}\left\n g\right\n_{L^2}+\left\n f\right\n_{L^{2\rho}}\left\n f\right\n_{L^{2\rho}})  \left\n h\right\n_{\widehat{L}^p}.
\end{align*}
Choose  $q_2$  as
 $\frac{\alpha}{2q_2} =  d \big( \frac{1}{2} - \frac{1}{2\rho} \big)$
 so that $(2q_2,2\rho)$ is an $\alpha$-fractional admissible pair.
 Then 
\begin{align}\label{d1}
\left\n J_{k,l,m}(t)\right\n_{\widehat{L}^p}&\lesssim\int_0^t\big(\left\n\psi_k(s)\right\n_{L^2}\left\n\psi_l(s)\right\n_{L^2}+\left\n\psi_k(s)\right\n_{L^{2\rho}}\left\n\psi_l(s)\right\n_{L^{2\rho}}\big)\left\n\psi_m(s)\right\n_{\widehat{L}^p}ds\\
&\lesssim t\left\n\psi_k\right\n_{L_t^\infty(L^2)}\left\n\psi_l\right\n_{L_t^\infty(L^2)}\left\n\psi_m\right\n_{L_t^\infty(\widehat{L}^p)}+\left\n\psi_k\right\n_{L_t^{2q_2}\left(L^{2\rho}\right)}\left\n\psi_l\right\n_{L_t^{2q_2}\left(L^{2\rho}\right)}\left\n\psi_m\right\n_{L_t^{{q_2}'}\left(\widehat{L}^p\right)}\nonumber
\end{align}
using H\"older inequality.
Therefore
\begin{align}\label{12a}
\left\n\Psi_{\phi_0, k}(\psi)(t)\right\n_{\widehat{L}^p}\lesssim\left\n\psi_{0,k}\right\n_{\widehat{L}^p}+Nb^3(T+T^{\frac{1}{q_2'}}).
\end{align}  
For $(\underline{q},\underline{r})\in\{({q_1},r),(2q_2,2\rho),(\infty,2)\},$ by Proposition \ref{fst}  we have 
\begin{align*}
\n \Psi_{\psi_0,k}(\mathbf{\psi})(t)\n_{L^{\underline{q}}(L^{\underline{r}})}&\lesssim \n\psi_{0,k}\n_{L^2}+ \sum_{l=1}^N\n (K*|\psi_l|^2)\psi_k\n_{L_T^{{q_1}'}(L^{r'})}+ \sum_{l=1}^N\n (K*(\psi_k\overline{\psi_l})\psi_l\n_{L_T^{{q_1}'}(L^{r'})}.
\end{align*}
Now we have
 $\frac{1}{{q_1}'}=\frac{4s-\gamma}{4s}+\frac{1}{{q_1}},\ \frac{1}{r'}=\frac{\gamma}{2d}+\frac{1}{r}\ \text{and}\ \frac{4s-\gamma}{4s}=\frac{2}{{q_1}}+\frac{2s-\gamma}{2s}.$
By H\"older and Hardy-Littlewood-Sobolev inequalities, as in the calculation \eqref{b1} we have
\begin{align*}
\left\n (K*(\psi_k\overline{\psi_l})\psi_m\right\n_{L_T^{{q_1}'}(L^{r'})}
&\leq\left\n\left\n|\cdot|^{-\gamma}*|\psi_k\overline{\psi_l}|\right\n_{L^{\frac{2d}{\gamma}}}\right\n_{L_T^{\frac{4s}{4s-\gamma}}}\left\n\psi_m\right\n_{L_T^{q_1}(L^r)}\\
&\lesssim\left\n\left\n|\psi_k\overline{\psi_l}|\right\n_{L^{\frac{2d}{2d-\gamma}}}\right\n_{L_T^{\frac{4s}{4s-\gamma}}}\left\n\psi_m\right\n_{L_T^{q_1}(L^r)}\\
&\leq T^{1-\frac{\gamma}{2s}}\left\n\psi_k\right\n_{L_T^{q_1}(L^r)}\left\n\psi_l\right\n_{L_T^{q_1}(L^r)}\left\n\psi_m\right\n_{L_T^{q_1}(L^r)}.
\end{align*}
 Combining the above two inequalities, we obtain 
\begin{align*}
\n\Psi_{\psi_0,k}(\mathbf{\psi})(t)\n_{L^{\underline{q}}(L^{\underline{r}})}&\lesssim\n\psi_{0,k}\n_{L^2}+ T^{1-\frac{\gamma}{2s}}\sum_{l=1}^N\left\n\psi_l\right\n_{L_T^{q_1}(L^r)}^2\left\n\psi_k\right\n_{L_T^{q_1}(L^r)}\lesssim\n\psi_{0,k}\n_{L^2}+ T^{1-\frac{\gamma}{2s}}Nb^3.
\end{align*}
  Choose $b=2c\n\psi_0\n_{(L^2\cap\widehat{L}^p)^N}$ and $T>0$ small enough so that    \eqref{12a} and the above inequality 
  imply  $\Psi_{\psi_0} (\psi) \in \mathcal{U}_b^T$. On the other hand for $u,v\in\mathcal{V}_b^T,$ using trilinearity of $\h_{a,\gamma}$, we have 
\begin{align}\label{12b}
\n\Psi_{\psi_0, k}(u)(t)-\Psi_{\psi_0, k}(v)(t)\n_{\widehat{L}^p}&\leq\sum_{l=1}^N\int_0^t\left\n\h_{a,\gamma}(u_l,u_l,u_k)-\h_{a,\gamma}(v_l,v_l,v_k)\right\n_{\widehat{L}^p}\nonumber\\
&\ \ \ +\sum_{l=1}^N\int_0^t\left\n\h_{a,\gamma}(u_k,u_l,u_l)-\h_{a,\gamma}(v_k,v_l,v_l)\right\n_{\widehat{L}^p}\nonumber\\
&\lesssim N(T+T^{\frac{1}{q_2'}})b^2d(u,v).
\end{align}
On the other hand again by using Proposition \ref{fst} and trilinearity we obtain
\begin{align}\label{7d}
\n \Psi_{\psi_0,k}(u)(t)-\Psi_{\psi_0,k}(v)(t)\n_{L^{\underline{q}}(L^{\underline{r}})}\lesssim T^{1-\frac{\gamma}{2s}}Nb^2d(u,v).
\end{align}
Choose $T>0$ further small so that   \eqref{12b}  and \eqref{7d} imply that $\Psi_{\psi_0}$ is a contraction.\\
$\bullet$ \textbf{Case II:} $0<\gamma<2d(\frac{1}{2}-\frac{1}{p})$, $2<p\leq\infty.$

Let $X$ be as in Proposition \ref{lhe}.  For  $b,T>0,$ let
$U_b^T=\{v\in L_T^\infty(X):\n v\n_{L_T^\infty(X)}\leq b\}.$  We set $\mathcal{U}_b^T=(U_b^T)^N$ and the distance on it by
$d(u,v)=\max\big\{\n u_j-v_j\n_{L_T^\infty(X)}:j=1,2,\cdots,N\big\},$
where $u,v\in\mathcal{U}_b^T$.  Next, we show that the mapping $\Psi_{\psi_0},$ defined by  \eqref{12}, takes $\mathcal{U}_b^T$ into itself for suitable choice of $b$ and small  $T>0$.  Let $\psi=(\psi_1,...,\psi_N)\in \mathcal{U}_b^T.$ 
We note that 
\begin{align*}
\left\n\Psi_{\psi_0, k}(\psi)(t)\right\n_{X}&\lesssim\left\n\psi_{0,k}\right\n_{X}+\sum_{l=1}^N\int_0^t\left\n\h_{a,\gamma}(\psi_l(s),\psi_l(s),\psi_k(s))\right\n_{X}ds\\
&\ \ \ +\sum_{l=1}^N\int_0^t\left\n\h_{a,\gamma}(\psi_k(s),\psi_l(s),\psi_l(s))\right\n_{X}ds.
\end{align*}
Therefore taking $b=2\left\n\psi_0\right\n_{X^N}$ and using Proposition \ref{lhe} \eqref{lhec}, we have
\begin{align*}
\left\n\Psi_{\psi_0, k}(\psi)(t)\right\n_{X}&\leq \frac{b}{2}+c\sum_{l=1}^N\int_0^t\left\n\psi_l\right\n_{X}^2\left\n\psi_k\right\n_{X}\leq \frac{b}{2}+cNTb^3\leq b.
\end{align*}for $T>0$ small enough.
For $u,v\in\mathcal{U}_b^T,$ by tri-linearity of $\h_{a,\gamma}$ we have
\begin{align*}
\left\n\Psi_{\psi_0, k}(u)(t)-\Psi_{\psi_0, k}(v)(t)\right\n_{X}&\lesssim Nb^2T d(u,v).
\end{align*}
Thus  $\Psi: \mathcal{U}_b^T \to \mathcal{U}_b^T$ is a  contraction  provided $T>0$ is small enough.
\end{proof}

%

We introduce the function  space $\w{Z}_{q,\theta}^p(T) \ (1\leq p \leq \infty)$ which is similar to $\w{Y}_{q,\theta}^p(T)$ to get local well-possedness. 
Specifically, we define 
$$\w{W}_{q,\theta}^p(T)=\big\{v:[0,T]\times\rd\rightarrow\C\ :\ \n v\n_{\w{W}_{q,\theta}^p(T)}=\n t^{\theta}\n(\partial_tv)(t,\cdot)\n_{\widehat{L}^p}\n_{L_T^q}<\infty\big\}$$
and
$$\widetilde{Z}_{q,\theta}^p(T)=\big\{v\in\widetilde{W}_q^p(T): \n v\n_{\widetilde{Z}_{q,\theta}^p(T)}:=\n v(0)\n_{\widehat{L}^p}+\n v\n_{\widetilde{W}_{q,\theta}^p(T)} < \infty  \big\}.$$
Now we state the required inclusion result.
\begin{Lemma}\label{inclu2}
Let $T>0$, $p\geq1$ and $0<\gamma<1$. Then
$\widetilde{Z}_{\infty,\theta}^p(T)\hookrightarrow C([0, T],\widehat{L}^p(\R^d)).$
\end{Lemma}
\begin{proof}[{\bf Proof}] The proof is similar to the proof of Lemma \ref{inclu1}. See \cite[Lemma 2.1]{hyakuna2019global} for details.
\end{proof}

\begin{Lemma}\label{10} Denote
\begin{align*}
\mathcal{D}_{a,\gamma}^\alpha(v_1,v_2,v_3)(t)=\int_0^tU_\alpha(-s)\h_{a,\gamma}(U_\alpha(s)v_1(s),U_\alpha(s)v_2(s),U_\alpha(s)v_3(s))ds.
\end{align*} 
\begin{enumerate}
\item\label{10a} Assume that $0<\gamma<\min\{\alpha,d(\frac{1}{p}-\frac{1}{2})\}$ when $1\leq p<2$ and $0<\gamma<\min\{\alpha,2d(\frac{1}{2}-\frac{1}{p})\}$ when $2<p\leq\infty$. Then 
$
\n \mathcal{D}_{a,\gamma}^\alpha(v_1,v_2,v_3)\n_{\widetilde{W}_{\alpha/\gamma,0}^2(T)\cap \w{W}_{\infty,0}^p}\lesssim\prod_{l=1}^3\n v_l\n_{\widetilde{Z}_{1,0}^p(T)\cap\widetilde{Z}_{1,0}^2(T)}.
$
\item\label{10b} Assume that  $2<p\leq\infty$ and $0<\gamma<2d(\frac{1}{2}-\frac{1}{p})$. Then 
$
\n \mathcal{D}_{a,\gamma}^\alpha(v_1,v_2,v_3)\n_{\widetilde{W}_{\infty,0}^2(T)\cap \w{W}_{\infty,0}^p}\lesssim\prod_{l=1}^3\n v_l\n_{\widetilde{Z}_{1,0}^p(T)\cap\widetilde{Z}_{1,0}^2(T)}.
$
\end{enumerate}
\end{Lemma}
\begin{proof}[{\bf Proof}] 
\eqref{10a} Set $q=\frac{\alpha}{\gamma},r=\frac{6d}{3d-2\gamma}$ so that $(3q,r)$ becomes an $\alpha$-fractional admissible pair.  By a similar argument as in the proof of Lemma \ref{8} \eqref{8a}, we obtain
\begin{align*}
\n\partial_t\mathcal{D}_{a,\gamma}^\alpha(v_1,v_2,v_3)\n_{L_T^q(L^2)}&\asymp\n U_\alpha(-t)\h_{a,\gamma}(U_\alpha(t)v_1(t),U_\alpha(t)v_2(t),U_\alpha(t)v_3(t))\n_{L_T^q(L^2)}\\
&=\n\h_{a,\gamma}(U_\alpha(t)v_1(t),U_\alpha(t)v_2(t),U_\alpha(t)v_3(t))\n_{L_T^q(L^2)}\\
&\lesssim\prod_{l=1}^3\n U(t)v_l\n_{L_T^{3q}(L^{\frac{6d}{3d-\gamma}})}\lesssim\prod_{l=1}^3\n v_l\n_{\w{Z}_{1,0}^2(T)}.
\end{align*}
 In view of Proposition \ref{lhe} and Lemma \ref{4}, we obtain\begin{align*}
\n\mathcal{D}_{a,\gamma}^\alpha(v_1,v_2,v_3)\n_{\w{W}_{\infty,0}^p(T)}&=\sup_{t\in[0,T]}\n\partial_t\mathcal{D}_{a,\gamma}^\alpha(v_1,v_2,v_3)(t)\n_{\widehat{L}^p}\\
&\asymp\sup_{t\in[0,T]}\n\h_{a,\gamma}(U_\alpha(t)v_1(t),U_\alpha(t)v_2(t),U_\alpha(t)v_3(t))\n_{\widehat{L}^p}\\
&\lesssim\sup_{t\in[0,T]}\prod_{l=1}^3\n v_l(t)\n_{L^2\cap\widehat{L}^p}\lesssim\prod_{l=1}^3\n v_l\n_{\w{Z}_{1,0}^2(T)\cap\w{Z}_{1,0}^p(T)}.
\end{align*}
\eqref{10b} Here it remains to estimate the $\widetilde{W}_{\infty,0}^2(T)$-semi norm which follows in a similar way as the $\widetilde{W}_{\infty,0}^p(T)$ estimate above.
\end{proof}

\begin{proof}[{\bf Second proof of Theorem \ref{lhw}}] 
For $\alpha\in(\frac{2d}{2d-1},2),$ we assume $d\geq2$ and initial data to be radial. \\
$\bullet$ \textbf{Case I:} $0<\gamma<\min\{\alpha,d(\frac{1}{p}-\frac{1}{2})\}$ when $1\leq p<2$ and $0<\gamma<\min\{\alpha,2d(\frac{1}{2}-\frac{1}{p})\}$ when $2<p\leq\infty$. 
 
Applying $U_\alpha(-t)$ to the Duhamel's formula, we rewrite  \eqref{hf}  as
\begin{align}\label{f2}
\phi_k(t)&=\psi_{0,k}+i\int_0^t U_\alpha(-s)(H_\psi\psi_k)(s)ds-i\int_0^t U_\alpha(-s)(F_\psi\psi_k)(s)ds=:\Phi_{\psi_0,k}(\phi).
\end{align}
For $b>b, T>0,$  let 
$V_b^T(v_0)=\big\{v\in\widetilde{Z}_{\alpha/\gamma,0}^2(T)\cap\widetilde{Z}_{\infty,0}^p(T)\ : \ \n v\n_{\widetilde{W}_{\alpha/\gamma,0}^2(T)\cap\widetilde{W}_{\infty,0}^p(T)}\leq b,\ v(0)=v_0\big\}.$   We set
$\mathcal{V}_b^T(\psi_0)=V_b^T(\psi_{0,1})\times V_b^T(\psi_{0,2})\times\cdots\times V_b^T(\psi_{0,N}),$
and  the distance on it by 
$d(u,v)=\max\big\{\n u_j-v_j\n_{\widetilde{W}_{\alpha/\gamma,0}^2(T)\cap\widetilde{W}_{\infty,0}^p(T)}:\ j=1,2,\cdots,N\big\}.$
Next, we show that the mapping  $\Phi_{\psi_0}$  defined by \eqref{f2} takes  $\mathcal{V}_b^T(\psi_0)$ into itself for suitable choice of  $b>0$ and small  $T>0.$ In fact,  taking $0<T<1$ and as the terms with integral sign in \eqref{f2} is combination of $\mathcal{D}_{a,\gamma}^\alpha(\phi_k,\phi_l,\phi_m)$'s, by  Lemma \ref{10}, we obtain
\begin{align*}
\n\Phi_{\psi_0,k}(\phi)\n_{\widetilde{W}_{\alpha/\gamma,0}^2(T)\cap\widetilde{W}_{\infty,0}^p(T)}
&\lesssim\n \phi_k\n_{\widetilde{Z}_{1,0}^2(T)\cap\widetilde{Z}_{1,0}^p(T)}\sum_{l=1}^N\n \phi_l\n_{\widetilde{Z}_{1,0}^2(T)\cap\widetilde{Z}_{1,0}^p(T)}^2.
\end{align*}
By H\"older inequality, we have $\n v_l\n_{\widetilde{W}_{1,0}^2(T)}\leq T^{1-\frac{\gamma}{\alpha}}\n v_l\n_{\widetilde{W}_{\alpha/\gamma,0}^2(T)}$ and $\n v_l\n_{\widetilde{W}_{1,0}^p(T)}\leq T\n v_l\n_{\widetilde{W}_{\infty,0}^p(T)}$. Therefore, for $0<T<1,$ we have 
\begin{equation}\label{11}
\n v_l\n_{\widetilde{Z}_{1,0}^p(T)\cap\widetilde{Z}_{1,0}^2(T)}\lesssim \n v_l(0)\n_{\widehat{L}^p\cap L^2}+T^{1-\frac{\gamma}{\alpha}}\n v_l\n_{\widetilde{W}_{\alpha/\gamma,0}^2(T)\cap\widetilde{W}_{\infty,0}^p(T)}.
\end{equation}
Hence
\begin{align*}
\n\Phi_{\psi_0,k}(\phi)\n_{\widetilde{W}_{\alpha/\gamma,0}^2(T)\cap\widetilde{W}_{\infty,0}^p(T)}&\lesssim N\n\psi_0\n_{L^2\cap \widehat{L}^p}^3+T^{3(1-\frac{\gamma}{\alpha})}\sum_{l=1}^N\n\phi_l\n_{\widetilde{W}_{\alpha/\gamma,0}^2(T)\cap\widetilde{W}_{\infty,0}^p(T)}^3.
\end{align*}
Set $b=2cN\n\psi_0\n_{(L^p\cap L^2)^N}^3$ then choose $T>0$ small enough to get $\n\Phi_{\psi_0,k}(\phi)\n_{\widetilde{W}_{2/\gamma,0}^2(T)\cap\widetilde{W}_{\infty,\gamma}^p(T)}\leq b.$ It follows that   $\Phi (\phi) \in \mathcal{V}_b^T(\psi_0).$
 For $u,v\in\mathcal{V}_b^T(\psi_0)$
 by trilinearity, Lemma \ref{10} and \eqref{11}
\begin{align*}
d\left(\Phi_{\psi_0,k}(u),\Phi_{\psi_0,k}(v)\right)&\lesssim T^{1-\frac{\gamma}{\alpha}}\left(\n\psi_0\n_{\widehat{L}^p\cap L^2}^2+T^{2(1-\frac{\gamma}{\alpha})}b^2\right)\sum_{l=1}^N\n u_l-v_l\n_{\widetilde{W}_{2/\gamma,0}^2(T)\cap\widetilde{W}_{\infty,0}^p(T)}\\
&\lesssim T^{1-\frac{\gamma}{\alpha}}N\left(\n\psi_0\n_{\widehat{L}^p\cap L^2}^2+T^{2(1-\frac{\gamma}{\alpha})}b^2\right)d\left(u,v\right).
\end{align*}
Then $\Phi_{\psi_0}: \mathcal{V}_b^T(\psi_0) \to \mathcal{V}_b^T(\psi_0)$ is a contraction for further small enough $T>0$ if needed.\\
$\bullet$  \textbf{Case II:}  $0<\gamma<2d(\frac{1}{2}-\frac{1}{p})$ with $2<p\leq\infty.$

Lemma \ref{10} \eqref{10b} and similar argument as in Case I give the solution $\phi\in(\widetilde{Z}_{\infty,0}^2(T)\cap\widetilde{Z}_{\infty,0}^p(T))^N$ to \eqref{f2}.  
\end{proof}

\begin{Remark}
Note that the second proof gives the existence of solution in Zhou spaces.  So it adds Zhou space regularity to the solutions given by the first proof. For $0<\gamma<1$ such solutions are in $C([0, T],\widehat{L}^p(\R^d))$ (see Lemma \ref{inclu2}).
\end{Remark}


\subsection{Global well-posedness in $L^p \cap L^2$}\label{pglobal1}


 We extend the local solution established in  Theorem \ref{lw}  globally. Let $\phi=(\phi_1,\phi_2,\cdots,\phi_N)$ be the local solution (given by Theorem \ref{lw})  to \eqref{t5} which is in $\left(C([0,T],L^p(\rd)\cap L^2(\rd))\right)^N$ for any $0<T<T_0$.  We start with the following lemma.

\begin{Lemma}\label{ul}
On the time interval $[0,T_0)$,  the local solution (given by Theorem \ref{lw}) $\psi(t)=(U(t)\phi_1(t),\cdots,U(t)\phi_N(t))$ coincides with the global $L^2$-solution for the initial datum $\psi_0=\psi(0)$ given by Proposition \ref{miF}.
\end{Lemma}
\begin{proof}[{\bf Proof}]  The assertion 
follows from uniqueness of local solution given by Theorem \ref{lw} (particularly see the metric space defined in the Case I of the First  proof of Theorem \ref{lw} in Subsection \ref{plw}) and Proposition \ref{miF}.
\end{proof}

\begin{proposition}\label{ga}
Assume $0<\gamma<d/2$. Let $T_0>0$ be such that for any $0<T<T_0$ the local solution $\phi$ of \eqref{6} exists in $C([0,T],L^p(\rd)\cap L^2(\rd))^N$. Then 
$
\sup_{t\in[0,T_0)}\n \phi(t)\n_{(L^p)^N}<\infty.
$
\end{proposition}
\begin{proof}[{\bf Proof}] 
We fix $T\in (0, T_0)$ and $t\in [0, T].$
Taking  \eqref{6} into account,  we obtain 
\begin{equation}\label{16}
\n\phi_k(t)\n_{L^p}\lesssim\n\psi_{0,k}\n_{L^p}+\sum_{l=1}^{N}\sum_{j=1}^2\n I_{l,l,k}^j(t)\n_{L^p}+\sum_{l=1}^{N}\sum_{j=1}^2\n I_{k,l,l}^j(t)\n_{L^p},
\end{equation}
where   $I_{k,l, k}^{j}$  are given by \eqref{inl}. Since the solution of (5.16) in $C([0,T],L^p(\rd)\cap L^2(\rd))^N$ for all $0<T<T_0$, we have a conservation of $L^2$-norm i.e. $\|\phi_k(t)\|_{L^2}=\|\phi_k(0)\|_{L^2}$.
Using Proposition \ref{le}\eqref{lea}, we have 
\begin{align}\label{17}
\n I_{k,l,m}^1(t)\n_{L^p}&\lesssim\int_0^ts^{-\gamma}\n\phi_k(s)\n_{L^2}\n\phi_l(s)\n_{L^2}\n\phi_m(s)\n_{L^p}ds\nonumber\\
&\lesssim T_0^{1-\gamma-1/q_4}\n\phi_k\n_{L_{T_0}^\infty(L^2)}\n\phi_l\n_{L_{T_0}^\infty( L^2)}\n\phi_m\n_{L_t^{q_4}(L^p)},
\end{align}
where $q_4$ is chosen so that $q_4>\frac{1}{1-\gamma}$.  Let $q_2$ and $\rho$  be  given as  in  \eqref{gr1} and \eqref{gr2} respectively.  By  \eqref{c2}, we have 
\begin{align*}
\n I_{k,l,m}^2(t)\n_{L^p} &\lesssim T_0^{d-\gamma-d/\rho}\n\psi_k\n_{L_{T_0}^{q_2}(L^{2\rho})}\n\psi_l\n_{L_{T_0}^{q_2}(L^{2\rho})}\n\phi_m\n_{L_t^{q_3}(L^p)}.
\end{align*}
 Note that $d-\gamma-\frac{d}{\rho}>0$. 
 Since $(q_2,2\rho)$ is an admissible pair,  in view of Lemma \ref{ul} and Proposition \ref{miF}, we have $\n\psi_j\n_{L_{T_0}^{q_2}(L^{2\rho})}<\infty$.  It follows that 
\begin{equation}\label{18}
\n I_{k,l,m}^2(t)\n_{L^p}\leq C_{T_0}\n\phi_m\n_{L_t^{q_3}(L^p)}.
\end{equation}
Thus we have from \eqref{16}, \eqref{17} and \eqref{18} that
\begin{equation}
\n\phi(t)\n_{(L^p)^N}\lesssim C_{\psi_0,T_0}+NC_{\psi_0,T_0}\n\phi\n_{L_t^{q}((L^p)^N)},
\end{equation}
where $q=\max\{q_3,q_4\}$. Therefore 
$$\n\phi(t)\n_{(L^p)^N}^q\lesssim C_{\psi_0,T_0}^q+N^qC_{\psi_0,T_0}^q\int_0^t\n\phi(t)\n_{(L^p)^N}^q dt.$$
By  Gronwall's lemma 
$\n\phi(t)\n_{(L^p)^N}^q\lesssim C_{\psi_0,T_0}^q\big(1+N^qC_{\psi_0,T_0}^qte^{C_{T_0,N,q}t}\big)$ which is desired.
\end{proof}

Let  $\psi(t)=(\psi_1(t),\cdots,\psi_N(t))$ be  a global  $L^2$-solution given by Proposition \ref{miF}.
We define $$T_+(\psi_0)=\sup\left\{T>0:U(-t)\psi(t)|_{[0,T]\times\rd}\in C\big([0,T],L^p(\rd)\big)^N\right\}$$ where $U(-t)\psi(t)=(U(-t)\psi_1(t),\cdots,U(-t)\psi_N(t))$.
By Theorem \ref{lw}, we have $T_+(\psi_0)>0$.

\begin{proposition}\label{gb}
 Assume $T_+(\psi_0)<\infty$. Then 
$\lim_{t\nearrow T_+(\psi_0)}\n U(-t)\psi(t)\n_{(L^p)^N}=\infty.$
\end{proposition}
\begin{proof}[{\bf Proof}]  We point out that the assertion relies on the fact that the  local existence time $T$,  from Theorem \ref{lw}, depends only on $\n\psi_0\n_{(L^2\cap L^p)^N},\gamma,d,N$.    Now the proof  is standard, see e.g,  \cite[Lemma 5.4]{hyakuna2018global} for the Hartree equation,  and so we omit the details.   
\end{proof}

\begin{proof}[{\bf Proof of Theorem \ref{global1}}]
It is enough to prove that $T_+(\psi_0)=\infty$. If not, Proposition \ref{gb} implies 
$\lim_{t\nearrow T_+(\psi_0)}\n U(-t)\psi(t)\n_{(L^p)^N}=\infty$
contradicting Proposition \ref{ga} as $T_+(\psi_0)>0$.
The last assertion of the theorem follows from Proposition \ref{miF} and Hausdorff-Young inequality.
\end{proof}

\subsection{Global well-posedness in $\widehat{L}^p\cap L^2$}\label{pglobal2}

\begin{proof}[{\bf Proof of Theorem \ref{global2}}]
The  proof strategy is similar to the proof of  Theorem \ref{global1}.  Specifically, taking Theorem \ref{lhw} and Proposition \ref{miF} into account, to prove  Theorem \ref{global2},  it is enough to show that the $(\widehat{L}^p)^N$-norm of the solution remains bounded in finite time. Let $t\in [0, T].$ \\ 
\noindent
$\bullet$ \textbf{Case I:} $0<\gamma<\min\{\alpha,\frac{d}{2}\}$, $1\leq p<2$.
 
  By  \eqref{12},  we have
\begin{align}\label{19}
\left\n\psi_k(t)\right\n_{\widehat{L}^p}&\leq\left\n\psi_{0,k}\right\n_{\widehat{L}^p}+\sum_{l=1}^N\int_0^t\left\n\h_{a,\gamma}(\psi_l(s),\psi_l(s),\psi_k(s)\right\n_{\widehat{L}^p}ds\\
&\ \ \ +\sum_{l=1}^N\int_0^t\left\n\h_{a,\gamma}(\psi_k(s),\psi_l(s),\psi_l(s)\right\n_{\widehat{L}^p}ds.\nonumber
\end{align}
By Propositions \ref{lhe}\eqref{lheb} and \ref{miF}, we have 
\begin{align*}
\int_0^t\left\n\h_{a,\gamma}(\psi_k(s),\psi_l(s),\psi_m(s)\right\n_{\widehat{L}^p}ds&\lesssim\int_0^t\left\n\psi_k(s)\right\n_{L^2}\left\n\psi_l(s)\right\n_{L^2}\left\n\psi_m(s)\right\n_{\widehat{L}^p\cap L^2}ds\\
&\lesssim\int_0^t\left\n\psi_m(s)\right\n_{\widehat{L}^p\cap L^2}ds=T\n\psi_{0,m}\n_{L^2}+\int_0^t\left\n\psi_m(s)\right\n_{\widehat{L}^p}ds.
\end{align*}
Using this and \eqref{19}, we have  \begin{align*}
\left\n\psi(t)\right\n_{(\widehat{L}^p)^N}&\lesssim\left\n\psi_0\right\n_{(\widehat{L}^p)^N}+NT\n\psi_{0}\n_{(L^2)^N}+N\int_0^t\left\n\psi(s)\right\n_{(\widehat{L}^p)^N}ds.
\end{align*}
Now the result follows by Gronwall's lemma.\\
\noindent 
$\bullet$ \textbf{Case II:} $0<\gamma<\min\{\alpha,\frac{d}{2}\}$ ($1\leq p\leq\infty$). 

For $\alpha\in(\frac{2d}{2d-1},2), $ we assume $d\geq2$  and initial data is radial.  By \eqref{f1} and  \eqref{d1},  we have 
\begin{align*}
\n J_{k,l,m}(t)\n_{\widehat{L}^p}&\lesssim T^{\frac{1}{q_2}}\left\n\psi_k\right\n_{L_{T}^\infty(L^2)}\left\n\psi_l\right\n_{L_{T}^\infty(L^2)}\left\n\psi_m\right\n_{L_t^{{q_2}'}(\widehat{L}^p)}+\left\n\psi_k\right\n_{L_{T}^{2q_2}\left(L^{2\rho}\right)}\left\n\psi_l\right\n_{L_{T}^{2q_2}\left(L^{2\rho}\right)}\left\n\psi_m\right\n_{L_t^{{q_2}'}\left(\widehat{L}^p\right)}.
\end{align*}
 By \eqref{19} and Strichartz estimates, we have 
$
\n\psi_k(t)\n_{\widehat{L}^p}\lesssim\n\psi_{0,k}\n_{\widehat{L}^p}+(1+T^{\frac{1}{q_2}})\sum_{l=1}^N\left\n\psi_l\right\n_{L_t^{q_2'}(\widehat{L}^p)}
$ and so $\n\psi(t)\n_{(\widehat{L}^p)^N}\lesssim C_{\psi_0,T}\big(1+N\n\psi\n_{L_t^{{q_2}'}((\widehat{L}^p)^N)}\big)$. Gronwall's lemma gives the bound.
\end{proof}

\subsection{Improved  well-posedness in 1D}\label{piglobal}

 We have proved  the result (local and global) if $0<\gamma<\frac{1}{2}$ for $d=1$, see Theorems \ref{lw} - \ref{global2}.  Now we improve it to  $0<\gamma<1$ for global existence. To do this it is enough to prove it for $\frac{1}{2}\leq \gamma<1$.  In the proof, we impose the condition
  $\frac{2}{3p}\leq \gamma<1$ for $p\in(\frac{4}{3},2]$ and $\frac{2}{3p'}\leq \gamma<1$ for $p\in[2,4)$. Note that $\frac{2}{3p}<\frac{1}{2}$ for $p\in(\frac{4}{3},2]$ and $\frac{2}{3p'}<\frac{1}{2}$ for $p\in[2,4)$. The extra ingredient we use here is Lemma \ref{31} below.  

\begin{Lemma}[Generalized Strichartz estimate  \cite{fefferman1970inequalities,cazenave2001note}] \label{31} 
  $\n U(t)\phi\n_{L^{3p}(\R\times\R)}\lesssim\n\phi\n_{\widehat{L}^p(\R)}$ for $\frac{4}{3}<p\leq2.$
 As a consequence, by  the  duality argument,   for  $2\leq p<4,$   $\sup_{I\subset J}\left\n\int_I U(-s)F(s)ds\right\n_{\widehat{L}^p}\leq \n F\n_{L^{(3p')'}(J\times\R)}.$ 
\end{Lemma}

\begin{proof}[{\bf Proof of Theorem \ref{iglobal}}] 
As an application of Lemma \ref{31}, we  shall obtain some improved  estimate for  $\|I_{k,l,m}\|_{L^p}$ (see \eqref{nl}).  Specifically, compare the estimate  \eqref{c1} to \eqref{e1} below. We shall see that this will play a vital role to improve the range of the exponent  $\gamma$ in the Hartree factor. 

$\bullet$ \textbf{Step A I:} Improving the local result for $L^p$ space.

Note that $\n\varphi\n_{L^p}=\n\F^{-1}\varphi\n_{\widehat{L}^{p'}}=\n\overline{\F\overline{\varphi}}\n_{\widehat{L}^{p'}}=\n\F\overline{\varphi}\n_{\widehat{L}^{p'}}$ and by \eqref{-1} $\F M_s\F^{-1}=U(-1/16\pi^2s)$ as
\begin{align*}
\F M_s\F^{-1}\varphi(\xi)&
=\int_{\rd}e^{-2\pi ix\cdot\xi}e^{i|x|^2/4s}\F^{-1}\varphi(x)dx=\int_{\rd}e^{2\pi ix\cdot\xi}e^{i|x|^2/4s}\F\varphi(x)dx\\
&=\int_{\rd}e^{2\pi ix\cdot\xi}e^{-4\pi^2i|x|^2(-1/16\pi^2s)}\F\varphi(x)dx=[U(-1/16\pi^2s)\varphi](\xi).
\end{align*}  In view of these and \eqref{nl}, we obtain $\left\n I_{k,l,m}(t)\right\n_{L^p}$ is less or equal to
\begin{align*}
&\left\n\int_0^ts^{-\gamma}M_s^{-1}\widehat{\h}_{a,\gamma,s}(M_s\phi_k(s),RM_s\phi_l(s),M_s\phi_m(s))ds\right\n_{L^p}\\
&=\left\n\int_0^ts^{-\gamma}\F\overline{M_s^{-1}\widehat{\h}_{a,\gamma,s}(M_s\phi_k(s),RM_s\phi_l(s),M_s\phi_m(s))}ds\right\n_{\widehat{L}^{p'}}\\
&=\left\n\int_0^ts^{-\gamma}\F M_s\overline{\widehat{\h}_{a,\gamma,s}(M_s\phi_k(s),RM_s\phi_l(s),M_s\phi_m(s))}ds\right\n_{\widehat{L}^{p'}}\\
&=\left\n\int_0^ts^{-\gamma}U(-1/16\pi^2s)\F\overline{\widehat{\h}_{a,\gamma,s}(M_s\phi_k(s),RM_s\phi_l(s),M_s\phi_m(s))}ds\right\n_{\widehat{L}^{p'}}\\
&=\left\n\int_{1/t}^\infty s^{\gamma-2}U\left(\frac{-s}{16\pi^2}\right)\F\overline{\widehat{\h}_{a,\gamma,1/s}(M_{1/s}\phi_k(1/s),RM_{1/s}\phi_l(1/s),M_{1/s}\phi_m(1/s))}ds\right\n_{\widehat{L}^{p'}}.
\end{align*}
Using Lemma \ref{31} and changing the $s$-variable, we get that
\begin{align*}
\left\n I_{k,l,m}(t)\right\n_{L^p}&\lesssim\left\n s^{\gamma-2}\F\overline{\widehat{\h}_{a,\gamma,1/s}(M_{1/s}\phi_k(1/s),RM_{1/s}\phi_l(1/s),M_{1/s}\phi_m(1/s))}\right\n_{L^{\w{r}}([16\pi^2/t,\infty)\times\R)}\\
&\leq\left\n s^{\gamma-2}\F\overline{\widehat{\h}_{a,\gamma,1/s}(M_{1/s}\phi_k(1/s),RM_{1/s}\phi_l(1/s),M_{1/s}\phi_m(1/s))}\right\n_{L^{\w{r}}([1/t,\infty)\times\R)}\\
&=\left\n s^{2-\gamma-2/\w{r}}\F\overline{\widehat{\h}_{a,\gamma,s}(M_s\phi_k(s),RM_s\phi_l(s),M_s\phi_m(s))}\right\n_{L^{\w{r}}((0,t]\times\R)},
\end{align*}
where $\w{r}=(3p)'$ . 
Since $U(-t)\psi_k(t)=\phi_k(t)$, by \eqref{tf} and  \eqref{nl1}, we have
$$\widehat{\h}_{a,\gamma,s}(M_s\phi_k(s),RM_s\phi_l(s),M_s\phi_m(s))=\left[(|\cdot|^{\gamma-1}*S_{a,s})\Omega(\psi_k(s),\psi_l(s))\right]*M_s\phi_m(s).$$ 
In view of this, we may obtain
\begin{align*}
&\left\n\F\overline{\widehat{\h}_{a,\gamma,s}(M_s\phi_l(s),RM_s\phi_l(s),M_s\phi_k(s))}\right\n_{L^{\w{r}}}\\
&\leq\left\n\F\left[(|\cdot|^{\gamma-1}*S_{a,s})\Omega(\psi_k(s),\psi_l(s))\right]\right\n_{L^{3p/2}}\left\n\F M_s\phi_m(s)\right\n_{L^{p'}}\\
&\lesssim\left\n|\cdot|^{-\gamma}*|\F\left[\Omega(\psi_k(s),\psi_l(s))\right]|\right\n_{L^{3p/2}}\left\n\phi_m(s)\right\n_{L^p}\lesssim\left\n\F\left[\Omega(\psi_k(s),\psi_l(s))\right]\right\n_{L^{\w{R}}}\left\n\phi_k(s)\right\n_{L^p},
\end{align*}where $\w{R}=\left(1+\frac{2}{3p}-\gamma\right)^{-1}.$
Using Lemma \ref{13} we have \begin{align*}
\left\n\F\overline{\widehat{\h}_{a,\gamma,s}(M_s\phi_l(s),RM_s\phi_l(s),M_s\phi_k(s))}\right\n_{L^{\w{r}}}&\lesssim|s|^{1-1/\w{R}}\left\n\psi_k(s)\right\n_{L^{2\w{R}}}\left\n\psi_l(s)\right\n_{L^{2\w{R}}}\left\n\phi_m(s)\right\n_{L^p}.
\end{align*}
Note that $3-\gamma-2/\w{r}-1/\w{R}=0$ and hence by H\"older's inequality
\begin{align}\label{e1}
\left\n I_{k,l,m}(t)\right\n_{L^p}&\lesssim\left\n\psi_k\right\n_{L_T^{\w{Q}}(L^{2\w{R}})}\left\n\psi_l\right\n_{L_T^{\w{Q}}(L^{2\w{R}})}\left\n\phi_m\right\n_{L_t^{2/(2-\gamma)}(L^p)}\\
&\lesssim T^{1-\frac{\gamma}{2}}\left\n\psi_k\right\n_{L_T^{\w{Q}}(L^{2\w{R}})}\left\n\psi_l\right\n_{L_T^{\w{Q}}(L^{2\w{R}})}\left\n\phi_m\right\n_{L_t^{\infty}(L^p)},\nonumber
\end{align} where $\w{Q}=\big(\frac{\gamma}{4}-\frac{1}{6p}\big)^{-1}.$ Note that $\w{Q},\w{R}\geq2$ by the conditions imposed on $\gamma$.
Let  $q_1=\frac{8}{\gamma}\ \text{and}\  r=\frac{4}{2-\gamma}.$
We  define 
\begin{align*}
V_b^T=\big\{v\in L_T^\infty(L^p(\rd)\cap L^2(\rd)): &\ \n v\n_{L_T^\infty(L^p\cap L^2)}\leq b,\\
&\ \n U(t)v(t)\n_{L_T^{\w{Q}}(L^{2\w{R}})}\leq b,\ \n U(t)v(t)\n_{L_T^{q_1}(L^r)}\leq b\big\}
\end{align*} 
and
$\mathcal{V}_b^T=(V_b^T)^N.$
Now arguing as  in   Case I  of the  proof of Theorem \ref{lw}, we  can establish the  local  well-posedness of  \eqref{hf} with $0<\gamma<2$   in $L^p(\R)\cap L^2(\R).$ 
\\
\noindent
$\bullet$  \textbf{Step A II:} Improving the global result for $0<\gamma<1$ for $L^p$ space.

Note that from \eqref{6} and \eqref{e1} we have 
\begin{align*}
\n\phi_k(t)\n_{L^p}&\lesssim\n\psi_{0,k}\n_{L^p}+\sum_{l=1}^N\left\n\psi_l\right\n_{L_T^{\w{Q}}(L^{2\w{R}})}\left\n\psi_l\right\n_{L_T^{\w{Q}}(L^{2\w{R}})}\left\n\phi_k\right\n_{L_t^{2/(2-\gamma)}(L^p)}\\
&\ \ \ +\sum_{l=1}^N\left\n\psi_k\right\n_{L_T^{\w{Q}}(L^{2\w{R}})}\left\n\psi_l\right\n_{L_T^{\w{Q}}(L^{2\w{R}})}\left\n\phi_l\right\n_{L_t^{2/(2-\gamma)}(L^p)}.
\end{align*}
Now $(\w{Q},2\w{R})$ being admissible,  Strichartz estimate  gives
$\n\phi_k(t)\n_{L^p}\lesssim\n\psi_{0,k}\n_{L^p}+\sum_{l=1}^N\left\n\phi_l\right\n_{L_t^{2/(2-\gamma)}(L^p)}.
$ 
Now we can proceed as before in Subsection \ref{pglobal1}.\\
$\bullet$ \textbf{Step B I:} Improving the local result for $\widehat{L}^p$-space.

Using Lemma \ref{31} we have that 
\begin{align*}
\left\n\int_0^tU(t-s)H_{\gamma,\psi}(\psi_k)(s)ds\right\n_{\widehat{L}^p}&\lesssim\sum_{l=1}^N\n \h_{a,\gamma}(\psi_l,\psi_l,\psi_k)\n_{L^r([0,t]\times\R)}.
\end{align*}
Now using  H\"older, Hausdorff-Young and Hardy-Littlewood-Sobolev
\begin{align*}
\n\h_{a,\gamma}(\psi_k(s),\psi_l(s),\psi_m(s))\n_{L^r}&\leq\left\n|\cdot|^{-\gamma}*|\psi_k(s)\overline{\psi_l(s)}|\right\n_{L^{\w{R}}}\n\psi_m\n_{\widehat{L}^p}\\
&\lesssim\left\n\psi_k(s)\overline{\psi_l(s)}\right\n_{L^R}\n\psi_m\n_{\widehat{L}^p}\leq\left\n\psi_k(s)\right\n_{L^{2R}}\left\n\psi_l(s)\right\n_{L^{2R}}\n\psi_m\n_{\widehat{L}^p},
\end{align*}
where $\w{R}=\frac{3p'}{2}$, $R=\big(\frac{5}{3}-\gamma-\frac{2}{3p}\big)^{-1}.$
Therefore H\"older's inequality in $t$-variable we have (recall $J_{k,l,m}$ from \eqref{f1}), we have
\begin{align}\label{e2}
\n J_{k,l,m}(t)\n_{\widehat{L}^p}&\lesssim\n \psi_k\n_{L^Q\left([0,T_0],L^{2R}\right)}\n \psi_l\n_{L^Q\left([0,T_0],L^{2R}\right)}\n \psi_m\n_{L^{2/(2-\gamma)}\left([0,t],\widehat{L}^p\right)}\\
&\lesssim T^{1-\frac{\gamma}{2}}\n \psi_k\n_{L^Q\left([0,T_0],L^{2R}\right)}\n \psi_l\n_{L^Q\left([0,T_0],L^{2R}\right)}\n \psi_m\n_{L_t^\infty(\widehat{L}^p)},\nonumber
\end{align}where $Q=\big(\frac{\gamma}{4}+\frac{1}{6p}-\frac{1}{6}\big)^{-1}.$ Note that $Q,R\geq2$ by the conditions imposed on $\gamma$.
Let $q_1=\frac{8}{\gamma}\ \text{and}\  r=\frac{4}{2-\gamma},$ and  for $T, b>0,$  introduce the space
$$U_b^T=\left\{v\in L_T^\infty\big(L^2(\rd)\cap\widehat{L}^p(\rd)\big):\n v\n_{L_T^\infty(L^2\cap\widehat{L}^p)}\leq b,\n v\n_{L_T^{q_1}(L^r)}\leq b, \n v\n_{L_T^{Q}(L^{2R})}\leq b\right\}.$$
Now we proceed as in Case I in Subsection \ref{plhw}.\\
\noindent
$\bullet$ \textbf{Step B II:} Improving the global result for $\widehat{L}^p$-space.

By \eqref{19} and \eqref{e2} we have 
\begin{align*}
\n\psi_k(t)\n_{\widehat{L}^p}&\lesssim\n\psi_{0,k}\n_{\widehat{L}^p}+\sum_{l=1}^N\n \psi_l\n_{L^Q\left([0,T],L^{2R}\right)}^2\n \psi_k\n_{L^{2/(2-\gamma)}\left([0,t],\widehat{L}^p\right)}\\
&\ \ \ +\sum_{l=1}^N\n \psi_k\n_{L^Q\left([0,T],L^{2R}\right)}\n \psi_l\n_{L^Q\left([0,T],L^{2R}\right)}\n \psi_l\n_{L^{2/(2-\gamma)}\left([0,t],\widehat{L}^p\right)}
\end{align*}
Strichartz estimate  gives
$\n\psi_k(t)\n_{\widehat{L}^p}\lesssim\n\psi_{0,k}\n_{\widehat{L}^p}+\sum_{l=1}^N\n \psi_l\n_{L^{2/(2-\gamma)}\left([0,t],\widehat{L}^p\right)}$ as $(Q,2R)$ is admissible.
Now we can proceed as before in Subsection \ref{pglobal2}.
\end{proof}

%
%
%
%
%
%
%

\section{Failure of $C^3$-smoothness in  mere $\widehat{L}^p$ spaces}\label{pill1}
\subsection{Proof of Theorem  \ref{ill1} }
\begin{proof}[Proof of Theorem \ref{ill1}]
\textcolor{black}{
It is known\footnote{In \cite{bourgain1997periodic}, Bourgain introduced this approach to establish failure of $C^3-$smoothness  for the solution map of KdV and mKdv,  see also \cite{tzvetkov1999remark}. Since then, many authors have used this approach,  see  for example,   \cite[Proposition 4.1]{choffrut2018ill} for cubic nonlinear half-wave equation.} that if the map $\mathcal{U}(t)$ is $C^{3}-$smooth at zero in $(\widehat{L}^p(\R^d))^N,$ then  the necessary condition is that there exist $C>0$ such that 
\begin{eqnarray}\label{bi}
\left\| \frac{\partial^3[\mathcal{U}(t)(\delta \psi_0)]}{\partial\delta^3}\bigg|_{\delta=0} \right \|_{(\widehat{L}^p)^N}\leq C\|\psi_0\|_{(\widehat{L}^p)^N}^3 \quad \text{for all} \  \psi_0\in (\widehat{L}^p(\R^d))^N.\ 
\end{eqnarray}
In the following we shall show that  estimate \eqref{bi} cannot hold with a constant $C$ independent of $\psi_0\in (\widehat{L}^p(\R^d))^N$.
To do this, consider the problem (with $K=|\cdot|^{-\gamma}$)
\begin{equation*}
i\partial_t\psi_k-(-\Delta)^{\alpha/2}\psi=\sum_{l=1}^N(K\ast|\psi_l|^2)\psi_k-(K\ast(\psi_k\bar{\psi_l}))\psi_l,\quad\psi_k(0)=\delta\psi_{0,k}
\end{equation*}where $\delta\geq0$, $\psi_0=(\psi_{0,1}, \cdots, \psi_{0, N})$, $k=1,2,..., N.$
By Duhamel's formula, we have 
\begin{equation}\label{z}
\psi_k(\delta)(t)=U_\alpha(t)\delta\psi_{0,k}-i\int_0^tU_\alpha(t-\tau)[(K\ast(\psi_l(\delta)\bar{\psi_l}(\delta))\psi_k(\delta)-(K\ast(\psi_k(\delta)\bar{\psi_l}(\delta))\psi_l(\delta)](\tau)d\tau
\end{equation}where Einstein's convention is used for summation\footnote{i.e. repeated indexed in a product are summed up: for example $a_lb_lc_k$ stands for $c_k\sum_{l=1}^Na_lb_l$}.  
Note that one has $$\mathcal{U}(t)(\delta\psi_0)=\psi(\delta,t)=\left(\psi_1(\delta)(t), \cdots, \psi_N(\delta)(t)\right):=\psi(\delta)(t).$$
Therefore, in order to show  an estimate \eqref{bi} fail (and hence failure of $C^3-$smoothness),  it is suffice 
to show that the following estimate fails
\begin{equation}\label{Y}
\left\|\frac{\partial^3\psi_k}{\partial\delta^3}(0,t)\right\|_{\widehat{L}^p}\leq C\|\psi_0\|_{(\widehat{L}^p)^N}^3
\end{equation}
 for atleast one $k$.  To this end,  we  shall first compute $\frac{\partial^3\psi_k}{\partial\delta^3}(0,t)$. Put $I_{k,l,m}(\delta,t,\cdot)=\int_0^tU_\alpha(t-\tau)[(K\ast(\psi_k(\delta)\bar{\psi_l}(\delta))\psi_m(\delta)](\tau)d\tau.$
 By straightforward calculations,  we get
\begin{align*}
\frac{\partial I_{k,l,m}}{\partial\delta}&=\int_0^tU_\alpha(t-\tau)[(K\ast(\psi_k\frac{\partial\bar{\psi_l}}{\partial\delta}+\frac{\partial\psi_k}{\partial\delta}\bar{\psi_l})))\psi_m+(K\ast(\psi_k\bar{\psi_l}))\frac{\partial\psi_m}{\partial\delta}](\tau)d\tau,\\
\frac{\partial^2 I_{k,l,m}}{\partial\delta^2}
&=\int_0^tU_\alpha(t-\tau)[(K\ast(\psi_k\frac{\partial^2\bar{\psi_l}}{\partial\delta^2}+2\frac{\partial\psi_k}{\partial\delta}\frac{\partial\bar{\psi_l}}{\partial\delta}+\frac{\partial^2\psi_k}{\partial\delta^2}\bar{\psi_l})))\psi_m\\
&\qquad+2(K\ast(\psi_k\frac{\partial\bar{\psi_l}}{\partial\delta}+\frac{\partial\psi_k}{\partial\delta}\bar{\psi_l})))\frac{\partial\psi_m}{\partial\delta}+(K\ast(\psi_k\bar{\psi_l}))\frac{\partial^2\psi_m}{\partial\delta^2}](\tau)d\tau,\\
\frac{\partial^3 I_{k,l,m}}{\partial\delta^3}
&=\int_0^tU_\alpha(t-\tau)[(K\ast(\psi_k\frac{\partial^3\bar{\psi_l}}{\partial\delta^3}+3\frac{\partial\psi_k}{\partial\delta}\frac{\partial^2\bar{\psi_l}}{\partial\delta^2}+3\frac{\partial^2\psi_k}{\partial\delta^2}\frac{\partial\bar{\psi_l}}{\partial\delta}+\frac{\partial^3\psi_k}{\partial\delta^3}\bar{\psi_l})))\psi_m\\
&\qquad+3(K\ast(\psi_k\frac{\partial^2\bar{\psi_l}}{\partial\delta^2}+2\frac{\partial\psi_k}{\partial\delta}\frac{\partial\bar{\psi_l}}{\partial\delta}+\frac{\partial^2\psi_k}{\partial\delta^2}\bar{\psi_l})))\frac{\partial\psi_m}{\partial\delta}\\
&\qquad+3(K\ast(\psi_k\frac{\partial\bar{\psi_l}}{\partial\delta}+\frac{\partial\psi_k}{\partial\delta}\bar{\psi_l})))\frac{\partial^2\psi_m}{\partial\delta^2}+(K\ast(\psi_k\bar{\psi_l}))\frac{\partial^3\psi_m}{\partial\delta^3}](\tau)d\tau.
\end{align*} 
Then 
 by the observation $\psi_k(0,t,\cdot)=0$ and using \eqref{z} we have
\begin{align*}
\frac{\partial\psi_k}{\partial\delta}(0,t,\cdot)&=U_\alpha(t)\psi_{0,k}-i(\frac{\partial I_{l,l,k}}{\partial\delta}-\frac{\partial I_{k,l,l}}{\partial\delta})(0,t,\cdot)=U_\alpha(t)\psi_{0,k},\\
\frac{\partial^2\psi_k}{\partial\delta^2}(0,t,\cdot)&=-i(\frac{\partial^2 I_{l,l,k}}{\partial\delta^2}-\frac{\partial^2 I_{k,l,l}}{\partial\delta^2})(0,t,\cdot)=0,\\
\frac{\partial^3\psi_k}{\partial\delta^3}(0,t,\cdot)&=-i(\frac{\partial^3 I_{l,l,k}}{\partial\delta^3}-\frac{\partial^3 I_{k,l,l}}{\partial\delta^3})(0,t,\cdot)\\
&=-6i\int_0^tU(t-\tau)[(K\ast(U_\alpha(\tau)\psi_{0,l}\overline{U_\alpha(\tau)\psi_{0,l}})U_\alpha(\tau)\psi_{0,k})\\
&\qquad-(K\ast(U_\alpha(\tau)\psi_{0,k}\overline{U_\alpha(\tau)\psi_{0,l}})U_\alpha(\tau)\psi_{0,l})]d\tau=:-6\mathcal{A}_k(t)(\psi_0).
\end{align*}
Then \eqref{Y} will fail for $k=1$ if \begin{equation}\label{Z}
\|\mathcal{A}_1(t)(\psi_0)\|_{\widehat{L}^p}\leq C\|\psi_0\|_{(\widehat{L}^p)^N}^3
\end{equation}fails. Below we adopt the technique from \cite{carles2014cauchy} (where the Hartree case is treated with $p=\infty$) to our scenario.}

\textcolor{black}{
Let $\psi_0=(\psi_{0,1},\psi_{0,2},0,\cdots0)\in \mathcal{S}(\rd)^N$ so that 
\[
\mathcal{A}_1(t)(\psi_0)=i\int_0^tU(t-\tau)[(K\ast|U_\alpha(\tau)\psi_{0,2}|^2)U_\alpha(\tau)\psi_{0,1})-(K\ast(U_\alpha(\tau)\psi_{0,1}\overline{U_\alpha(\tau)\psi_{0,2}})U_\alpha(\tau)\psi_{0,2})]d\tau.
\] }
Define a family $\{\psi_0^h\}_{h>0}$ of functions by
$$\psi_0^h(x)=h^\lambda\psi_0(h x),\quad (x\in\rd, \lambda>0).$$
For all $h>0$,
$$\n\psi_{0,k}^h\n_{\widehat{L}^p}=h^{\lambda-d/p}\n\psi_{0,k}\n_{\widehat{L}^p}.$$
Choose  $\lambda=d/p$  so that  $$\n\psi_{0,k}^h\n_{\widehat{L}^p}=\n\psi_{0,k}\n_{\widehat{L}^p} \quad \text{for all} \quad h>0,$$  and thus RHS of \eqref{Z} remain constant  for all $h>0$.

Next, we develop the expression of  $\| \mathcal{A}_1(\psi_0)(t)\|_{\widehat{L}^p}$ and  will show that  $\| \mathcal{A}_1(\psi_0)(t)\|_{\widehat{L}^p}\to\infty$ as $h\to0$, contradicting \eqref{Z}. 
We  note that   (with $c_\alpha=-(2\pi)^\alpha$)
\begin{align}\label{25}
&\F\left[\left(K*\left(U_\alpha(s)\psi_{0,k}^h(s)\overline{U_\alpha(s)\psi_{0,l}^h}\right)\right)U_\alpha(s)\psi_{0,m}^h\right](\xi)\nonumber\\
&\asymp\int_{\rd}\widehat{K}(\xi-y)\left(\F\left[U_\alpha(s)\psi_{0,k}^h\right]*\F\left[\overline{U_\alpha(s)\psi_{0,l}^h}\right]\right)(\xi-y)\F \left[U_\alpha(s)\psi_{0,m}^h\right](y)dy\nonumber\\
&=\int_{\R^{2d}}e^{ic_\alpha s \left(|y|^\alpha+|\xi-y-z|^\alpha-|z|^\alpha\right)}\widehat{K}(\xi-y)\F\psi_{0,k}^h(\xi-y-z)\F\overline{\psi_{0,l}^h}(z)\F \psi_{0,m}^h(y)dydz\\
&=h^{3(\lambda-d)}\int_{\R^{2d}}e^{ic_\alpha s \left(|y|^\alpha+|\xi-y-z|^\alpha-|z|^\alpha\right)}\widehat{K}(\xi-y)\widehat{\psi_{0,k}}(\frac{\xi-y-z}{h})\widehat{\overline{\psi_{0,l}}}(\frac{z}{h})\widehat{\psi_{0,m}}(\frac{y}{h})dydz.\nonumber
\end{align}
We rewrite that \begin{equation}\label{b}
\mathcal{A}_1(\psi_0)(t)=\int_0^t U_\alpha(t-s) g(s)ds=:F(t)
\end{equation}with
\begin{equation}\label{g3}
g(s)=-i\left(K*|U_\alpha(s)\psi_{0,2}|^2\right)U_\alpha(s)\psi_{0,1}+i\left(K*\left((U_\alpha(s)\psi_{0,1}\overline{U_\alpha(s)\psi_{0,2}}\right)\right)U_\alpha(s)\psi_{0,2}.
\end{equation}
Then
\begin{align*}\label{26}
&\n \mathcal{A}_1(\psi_0^h)(t)\n_{\widehat{L}^p}^{p'}=\n \F \mathcal{A}_1(\psi_0^h)(t)\n_{L^{p'}}^{p'}\\
&=\int_{\rd}\bigg|\int_0^t\F\left[ U_\alpha(t-s)\left[\left(K*|U_\alpha(s)\psi_{0,2}^h|^2\right)U_\alpha(s)\psi_{0,1}^h\right]\right](\xi)ds\nonumber\\
&\hspace{40pt}-\int_0^t\F\left[ U_\alpha(t-s)\left[\left(K*\left((U_\alpha(s)\psi_{0,1}^h\overline{U_\alpha(s)\psi_{0,2}^h}\right)\right)U_\alpha(s)\psi_{0,2}^h\right]\right](\xi)ds\bigg|^{p'}d\xi\nonumber\\
&=\int_{\rd}\bigg|\int_0^te^{ic_\alpha (t-s)|\xi|^\alpha}\F\left[\left(K*|U_\alpha(s)\psi_{0,2}^h|^2\right)U_\alpha(s)\psi_{0,1}^h\right](\xi)ds\\
&\hspace{40pt}-\int_0^te^{ic_\alpha (t-s)|\xi|^\alpha}\F\left[\left(K*\left(U_\alpha(s)\psi_{0,1}^h\overline{U_\alpha(s)\psi_{0,2}^h}\right)\right)U_\alpha(s)\psi_{0,2}^h\right](\xi)ds\bigg|^{p'}d\xi\nonumber =: I.
\end{align*} 
Performing the change of variables ($y\mapsto hy,  z\mapsto hz, s\mapsto s/h^{\alpha}$), we obtain
\begin{align*}
&\int_0^te^{ic_\alpha (t-s)|\xi|^\alpha}\F\left[\left(K*\left(U_\alpha(s)\psi_{0,k}^h\overline{U_\alpha(s)\psi_{0,l}^h}\right)\right)U_\alpha(s)\psi_{0,m}^h\right](\xi)ds\\
&=\int_0^te^{ic_\alpha (t-s)|\xi|^\alpha}h^{3(\lambda-d)}\int_{\R^{2d}}e^{ic_\alpha s \left(|y|^\alpha+|\xi-y-z|^\alpha-|z|^\alpha\right)}\widehat{K}(\xi-y)\\
&\hspace{145pt}\widehat{\psi_{0,k}}(\frac{\xi-y-z}{h})\widehat{\overline{\psi_{0,l}}}(\frac{z}{h})\widehat{\psi_{0,m}}(\frac{y}{h})dydzds\\
&=h^{3\lambda-d-\alpha}\int_0^{th^{\alpha}}e^{ic_\alpha (t-h^{-\alpha}s)|\xi|^\alpha}\int_{\R^{2d}}e^{ic_\alpha s \left(|y|^\alpha+h^{-\alpha}|\xi-h(y+z)|^\alpha-|z|^\alpha\right)}\widehat{K}(\xi-h y)\\
&\hspace{145pt}\widehat{\psi_{0,k}}(\frac{\xi}{h}-y-z)\widehat{\overline{\psi_{0,l}}}(z)\widehat{\psi_{0,m}}(y)dydzds.
\end{align*}
In view of this  we may  rewrite 
\begin{align*}
 I=& \int_{\rd}\bigg|h^{3\lambda-d-\alpha}\int_0^{th^{\alpha}}e^{ic_\alpha (t-h^{-\alpha}s)|\xi|^\alpha}\int_{\R^{2d}}e^{ic_\alpha s \left(|y|^\alpha+h^{-\alpha}|\xi-h(y+z)|^\alpha-|z|^\alpha\right)}\widehat{K}(\xi-h y)\\
&\hspace{45pt}\left(\widehat{\psi_{0,2}}(\frac{\xi}{h}-y-z)\widehat{\psi_{0,1}}(y)-\widehat{\psi_{0,1}}(\frac{\xi}{h}-y-z)\widehat{\psi_{0,2}}(y)\right)\widehat{\overline{\psi_{0,2}}}(z)dydzds\bigg|^{p'}d\xi\\
=&\ h^{d+(3\lambda-d-\alpha)p'}\int_{\rd}\bigg|\int_0^{th^{\alpha}}e^{ic_\alpha (th^{\alpha}-s)|\xi|^\alpha}\int_{\R^{2d}}e^{ic_\alpha s \left(|y|^\alpha+|\xi-y-z)|^\alpha-|z|^\alpha\right)}\widehat{K}(h(\xi-y))\\
&\hspace{45pt}\left(\widehat{\psi_{0,2}}(\xi-y-z)\widehat{\psi_{0,1}}(y)-\widehat{\psi_{0,1}}(\xi-y-z)\widehat{\psi_{0,2}}(y)\right)\widehat{\overline{\psi_{0,2}}}(z)dydzds\bigg|^{p'}d\xi\\
=&\ h^{d+(3\lambda-d-\alpha+(\gamma-d))p'}\int_{\rd}\bigg|\int_0^{th^{\alpha}}e^{c_\alpha i(th^{\alpha}-s)|\xi|^\alpha}\int_{\R^{2d}}e^{ic_\alpha s \left(|y|^\alpha+|\xi-y-z)|^\alpha-|z|^\alpha\right)}\widehat{K}(\xi-y)\\
&\hspace{45pt}\left(\widehat{\psi_{0,2}}(\xi-y-z)\widehat{\psi_{0,1}}(y)-\widehat{\psi_{0,1}}(\xi-y-z)\widehat{\psi_{0,2}}(y)\right)\widehat{\overline{\psi_{0,2}}}(z)dydzds\bigg|^{p'}d\xi
\end{align*}
as  the  kernel $K$  is homogeneous of degree $-\gamma$  (as $a=0$). Using \eqref{25}, we have
\begin{align*}
I=& h^{d+(3\lambda-2d-\alpha+\gamma)p'}\int_{\rd}\bigg|\int_0^{th^{\alpha}}e^{ic_\alpha (th^{\alpha}-s)|\xi|^\alpha}\F\left[\left(K*|U_\alpha\psi_{0,2}(s)|^2\right)U_\alpha(s)\psi_{0,1}(s)\right](\xi)ds-\\
&\hspace{60pt}-\int_0^{th^{\alpha}}e^{ic_\alpha (th^{\alpha}-s)|\xi|^\alpha}\F\left[\left(K*\left(U_\alpha(s)\psi_{0,1}\overline{U_\alpha(s)\psi_{0,2}}\right)\right)U_\alpha(s)\psi_{0,2}\right](\xi)ds\bigg|^{p'}d\xi\\
=&\ h^{d+(3\lambda-2d-\alpha+\gamma)p'}\n \mathcal{A}_1(\psi_0)(th^{\alpha})\n_{\widehat{L}^p}^{p'}.
\end{align*}
Since $\lambda=d/p$, combining  the above equalities,  we obtain 
  \begin{equation}\label{27}
\n \mathcal{A}_1(\psi_0^h)(t)\n_{\widehat{L}^p}\asymp h^{2d/p-d-\alpha+\gamma}\n \mathcal{A}_1(\psi_0)(th^{\alpha})\n_{\widehat{L}^p}.
\end{equation}
Next we investigate  more closely the quantity $F$ given by \eqref{b}.
Taylor formula gives
$$F(t)=F(0)+F'(0)t+\frac{t^2}{2}\int_0^1(1-\theta)F''(t\theta)d\theta.$$
Note that $F(0)=0$ and hence for $0\leq t\leq1,$ we have 
\begin{align}\label{dp1}
\n F(t)-F'(0)t\n_{\widehat{L}^p}\leq t^2\int_0^1\n F''(s\theta)\n_{\widehat{L}^p}d\theta\leq t^2\n F''\n_{L^\infty\left([0,1],\widehat{L}^p\right)}.
\end{align}
By Leibniz integral rule and Lemma \ref{g1} below,  the first derivative of $F$ is given by \begin{align*}
F'(t)&=U_\alpha(0)g(t)+\int_0^t\frac{\partial}{\partial t}U_\alpha(t-s)g(s)ds=g(t)-i\int_0^tU_\alpha(t-s)(-\Delta)^{\alpha/2}g(s)ds
\end{align*}  
and similarly the second derivative of $F$ is given by
\begin{align*}
F''(t)=g'(t)-i(-\Delta)^{\alpha/2}g(t)-\int_0^tU_\alpha(t-s)(-\Delta)^\alpha g(s)ds.
\end{align*}
Hence, we have 
\begin{align}\label{dp2}
\n F''\n_{L^\infty\left([0,1],\widehat{L}^p\right)}&\leq \n g'\n_{L^\infty\left([0,1],\widehat{L}^p\right)}+\n(-\Delta)^{\alpha/2}g\n_{L^\infty\left([0,1],\widehat{L}^p\right)}+\n(-\Delta)^\alpha g\n_{L^\infty\left([0,1],\widehat{L}^p\right)} < \infty, \nonumber
\end{align}
as  $\psi_0\in\mathcal{S}(\rd)^N$.   Using \eqref{dp1} and the above, we have 
\begin{align*}
\n F(t)-F'(0)t\n_{\widehat{L}^p}\lesssim t^2
\end{align*}
Using this and  $F'(0)=g(0),$ we obtain 
$t\n g(0)\n_{\widehat{L}^p}\lesssim\n F(t)\n_{\widehat{L}^p} +t^2.$
Hence in particular
$$th^{\alpha}\n g(0)\n_{\widehat{L}^p}\lesssim\n \mathcal{A}_1(\psi_0)(th^{\alpha})\n_{\widehat{L}^p} +t^2h^{2\alpha}$$
and so by \eqref{27}
\begin{align*}
\n \mathcal{A}_1(\psi_0^h)(t)\n_{\widehat{L}^p}&\asymp h^{2d/p-d-\alpha+\gamma}\n \mathcal{A}_1(\psi_0)(th^{\alpha})\n_{\widehat{L}^p}\gtrsim th^{2d/p-d+\gamma}\n g(0)\n_{\widehat{L}^p}-t^2h^{2d/p-d+\alpha+\gamma}.
\end{align*}
Now $2d/p-d+\gamma<0$ 
  if $\gamma<d-2d/p=2d(\frac{1}{2}-\frac{1}{p})$. Now using Lemma \ref{g2} below choose $\psi_{0,1},\psi_{0,2}\in\mathcal{S}(\rd)$ such that $\n g(0)\n_{\widehat{L}^p}\neq0$. Then since $\alpha>0,$ we have 
\begin{align*}
\n \mathcal{A}_1(\psi_0^h)(t)\n_{\widehat{L}^p}&\gtrsim th^{2d/p-d+\gamma}\n g(0)\n_{\widehat{L}^p}-t^2h^{2d/p-d+\gamma+\alpha}\longrightarrow\infty
\end{align*}
as $h\rightarrow0$. 
This completes the proof.
\end{proof}
\begin{center}
\begin{figure}
\begin{tikzpicture}[scale=.5]
\draw[dashed,gray] (-7,0)--(7,0) node[anchor=west] {\tiny{}};
\draw[dashed,gray] [->](0,-.5)--(0,3) node[anchor=north east] {\tiny{$ $}};
\draw[gray] [-] (-2,0) to [out=0,in=180] (0,2);
\draw[gray] [-] (0,2) to [out=0,in=180] (2,0);
\draw [-] (-3,0) to [out=0,in=180] (0,2);
\draw [-] (0,2) to [out=0,in=180] (3,0);
\draw[gray][->](2,0)--(7,0);
\draw[gray][->](-2,0)--(-7,0);
\draw[->](3,0)--(7,0);
\draw[->](-3,0)--(-7,0);
\draw[dashed,gray] (-6,2)--(6,2)node[anchor=west]{\textcolor{black}{\tiny{$1$}}};
\end{tikzpicture}
\caption{{Graphs of $\psi_{0,1}$, $\psi_{0,2}$ are indicated in gray and black curves respectively. 
}}\label{fig2}
\end{figure}
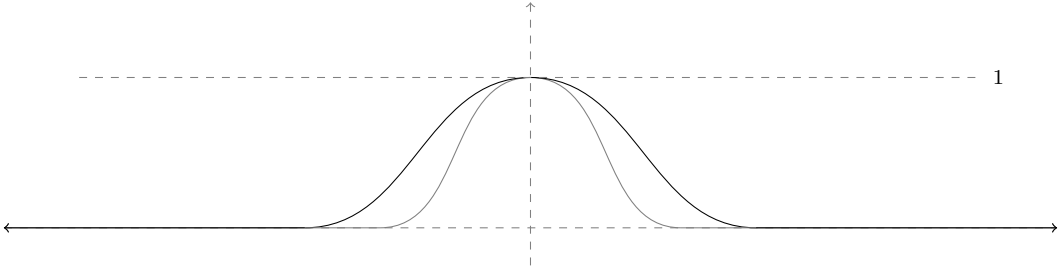
\end{center}
\begin{Remark}\label{g5}
In the above we presented proof for Hartree Fock case. 
For the reduced Hartree Fock case fixing $\psi_{0,2}=0$, the $g$ defined in \eqref{g3} would be replaced by $g(s)=-i\left(K*|U_\alpha(s)\psi_{0,1}|^2\right)U_\alpha(s)\psi_{0,1}$ 
 and hence Lemma \ref{g2} below would be fulfilled for any $\psi_{0,1}\neq0$.
 \end{Remark}

 \begin{Lemma}\label{g2}
Let $g$ be as in \eqref{g3}. Then there exist $\psi_{0,1},\psi_{0,2}\in\mathcal{S}(\rd)$ such that $\n g(0)\n_{\widehat{L}^p}\neq0$.
\end{Lemma}
\begin{proof}
Note that
$ig(0)=\left(K*|\psi_{0,2}|^2\right)\psi_{0,1}-\left(K*\left((\psi_{0,1}\overline{\psi_{0,2}}\right)\right)\psi_{0,2}\in \mathcal{S}(\rd)$ for all $\psi_{0,1},\psi_{0,2}\in\mathcal{S}(\rd)$. Let us choose real valued $\psi_{0,1},\psi_{0,2}\in\mathcal{S}(\rd)$ (see Figure \ref{fig2}) such that
$\psi_{0,1}(0)=\psi_{0,2}(0)=1$ ,
 $0\leq\psi_{0,1}\leq\psi_{0,2}$,
 $\psi_{0,1}<\psi_{0,2}$ on a set of positive measure in $\rd$.
Then \begin{align*}
ig(0)(0)=\int_{\rd}K(-y)(\psi_{0,2}(y)-\psi_{0,1}(y))\psi_{0,2}(y)dy>0.
\end{align*}
Thus $g(0)\neq0$ and hence $\F g(0)\neq0$ in $\mathcal{S}(\rd)$. This completes the proof.
\end{proof}

\begin{Lemma}\label{g1}
Let $f\in\mathcal{S}(\rd\times\R)$. Then for all $s,t\in\R,$ we have
$\frac{\partial}{\partial t}U_\alpha(t-s)f(s)=-iU_\alpha(t-s)(- \Delta)^{\alpha/2}f(s).$
\end{Lemma}
\begin{proof}[{\bf Proof}]
Let $v_s(t)=U_\alpha(t)f(s) \ (s\in \R).$  Then $v_s$ solves
\begin{equation*}
 \partial_t v_s+i(- \Delta)^{\alpha/2}v_s=0,\quad
v_{s|t=0}=f(s).
\end{equation*}
Hence $\partial_t v_s(t)=-i(- \Delta)^{\alpha/2}v_s(t)$ for all $t\in\R$. Therefore
\begin{align}\label{a}
\frac{\partial}{\partial t}U_\alpha(t-s)f(s)&=\partial_t v_s(t-s)=-i(- \Delta)^{\alpha/2}v_s(t-s)=-i(- \Delta)^{\alpha/2}U_\alpha(t-s)f(s).
\end{align}
 Note that the operators $(-\Delta)^{\alpha/2}$ and $U_{\alpha}$ commute.  Indeed,  for $h\in\mathcal{S}(\rd),$ we have
\begin{align*}
\F\left[(- \Delta)^{\alpha/2}U_\alpha(t)h\right]&=(2\pi)^\alpha|\xi|^\alpha\F\left[U_\alpha(t)h\right]=(2\pi)^\alpha|\xi|^\alpha e^{-(2\pi)^\alpha it|\xi|^\alpha}\widehat{h}\\
&=e^{-(2\pi)^\alpha it|\xi|^\alpha}\F\left[(- \Delta)^{\alpha/2}h\right]=\F\left[U_\alpha(t)(- \Delta)^{\alpha/2}h\right]
\end{align*}which implies $(- \Delta)^{\alpha/2}U_\alpha(t)h=U_\alpha(t)(- \Delta)^{\alpha/2}h$. Now \eqref{a} gives the result.
\end{proof}
\subsection{Further Remarks}\label{fr} In this subsection we aim to  discuss several issues (as promised in Subsection \ref{shf}) for the ill-posednes for (\#).  For the convenience of the reader,  we briefly recall  ill-posedness strategy developed in \cite{bejenaru2006sharp}.
We consider the abstract equation
\begin{equation}\label{20}
u=Lf+\mathcal{N}(u,u,u).
\end{equation}
Here,   $f\in D$ (initial data space) and $u\in S$ (solution space). 
The linear operator $L:D\rightarrow S$ is densely defined,  and  tri-linear operator $\mathcal{N}:S \times S\times S\rightarrow S$ is also densely defined. 

 The Hartree-Fock equation \eqref{hf} maybe written as 
\[\psi_k(t)=U_{\alpha}(t)\psi_{0,k}+i\int_0^t U_\alpha(t-s)(H_\psi\psi_k)(s)ds-i\int_0^t U_\alpha(t-s)(F_\psi\psi_k)(s)ds\]
for $k=1,2,...,n.$ In this case,  we may take
\[L=(U_\alpha,\cdots,U_\alpha) \quad \text{and} \quad \mathcal N=(\mathcal N^1,\cdots,\mathcal N^N)\]
where,  for $k=1,2,..., N,$ we put
$$\mathcal N^k(\psi,\psi,\psi)=i\int_0^t U_\alpha(t-s)(H_\psi\psi_k)(s)ds-i\int_0^t U_\alpha(t-s)(F_\psi\psi_k)(s)ds.$$ 
Thus,  \eqref{hf} can be rewritten as in the form  \eqref{20} namely \begin{equation*}\label{P2}
\psi=L\psi_0+\mathcal N(\psi,\psi,\psi)
\end{equation*} where $\psi=(\psi_1,\cdots, \psi_{N})$ and $\psi_0=(\psi_{0,1}, \cdots \psi_{0,N})$.
We say  \eqref{20} is \textbf{quantitative well-posed} in the Banach  spaces $D, S$ if
the following estimates
\[\|L(f)\|_{S} \lesssim  \|f\|_{D} \quad \text{and} \quad \|\mathcal{N}(u_1, u_2, u_3)\|_{S}\lesssim \|u_1\|_{S}\|u_2\|_{S}\|u_3\|_{S}. \]

We define the non-linear  maps (Picard's iterations) $A_n:D\to S$ for $n=1,2,...$ by the recursive formulae
\[A_1(f)=L(f)\]
\[A_n(f)=\sum_{n_1,n_2,n_3\geq 1 \atop{n_1+n_2+n_3=n}}\mathcal{N}(A_{n_1}(f),  A_{n_2}(f), A_{n_3}(f)) \quad \text{for} \ n>1.\]
The power series $\sum_{n=1}^{\infty}A_n(f)$   gives 
 solution to \eqref{20},  i.e.,  $u[f]=\sum_{n=1}^{\infty}A_n(f)$ whenever \eqref{20} is quantitative well-posed in the Banach  spaces $D, S.$ See \cite[Section 3]{bejenaru2006sharp} for details.

\begin{proposition}[Proposition 1 in \cite{bejenaru2006sharp}]\label{P1}
Suppose that \eqref{20} is quantitatively  well-posed in the Banach spaces $D, S$, with a solution map $f\mapsto u[f]$ from a ball $B_D$ in $D$ to a ball $B_S$ in $S$. Suppose that these spaces are then given other norms $D'$ and $S'$, which are weaker than $D$ and $S$ in the sense that
\[\|f\|_{D'} \lesssim\|f \|_D \quad \text{and} \quad \|u\|_{S'}\lesssim\|u\|_S.\]
  Suppose that the solution map $f\mapsto u[f ]$ is continuous from $(B_D , \|\|_{D'})$ (i.e. the ball $B_D$ equipped with the $D'$ topology) to $(B_S , \|\|_{S'})$ . Then for each $n$, the non-linear operator $A_n : D \to S$ is continuous from $(B_D,\|\|_{D'})$ to $(S, \|\|_{S'} )$. 
\end{proposition}

It's worth noting that if \eqref{20} is quantitatively well-posed from $D$ to $S$, Proposition \ref{P1} allows us to prove ill-posedness in coarse topologies $D',S'$ by demonstrating that at least one of the $A_n$ is discontinuous from $D'$ to $S'$.  We are now ready make several comments on proof of Theorem \ref{ill1}.

\begin{enumerate}
\item  In the proof of Theorem \ref{ill1},  for $k=1,2$, $\n \psi_{0,k}^h\n_{L^2}=h^{d/p- d/2}\n\psi_{0,k}\n_{L^2}\to\infty$ as $h\to0$.   Hence  data sequence $\{\psi_{0,k}^h\}$  leaves any ball in $\widehat{L}^p\cap L^2(\rd)$. Therefore, one cannot apply  Proposition \ref{P1} with this choice of data sequence to conclude that the solution map is discontinuous in a ball in $(\widehat{L}^p\cap L^2(\rd))^N$ with $\|\|_{(\widehat{L}^p)^N}$-topology.  However, recently we have studied on stronger form of ill-posedness (norm inflation with infinite loss of regularity) in Fourier Lebesgue spaces with negative regularity,  see \cite{bhimani2021strong}.
\item \label{21} In the proof of Theorem \ref{ill1},  $\mathcal{A}=(\mathcal{A}_1,\cdots,\mathcal{A}_N)$ is actually the third Picard iterate $A_3$ (upto a nonzero constant) and  we showed that estimate \eqref{Z} fails.  Hence,  by the  method of contradiction and using Theorem 3 in \cite{bejenaru2006sharp}, it follows that (\#) is not quantitatively well-posed form $(\widehat{L}^p)^N$ to $S$ for any subspace $S\hookrightarrow C([0,T),(\widehat{L}^p)^N)$ for any $T>0$. In particular, the estimates  $\|\mathcal{N}(\psi_1, \psi_2, \psi_3)\|_{C([0,T),(\widehat{L}^p)^N)}\lesssim \prod_{j=1}^3\|\psi_j\|_{C([0,T),(\widehat{L}^p)^N)}$ fails.
\item In view of Lemma \ref{lhe} \eqref{lhec}, (\#) is qualitatively well-posed from $(\widehat{L}^p\cap L^2)^N$ to $C([0,T),(\widehat{L}^p\cap L^2)^N)$.  On the other hand,  above point \eqref{21} says this fails if one works with mere $\widehat{L}^p$ instead of $\widehat{L}^p\cap L^2$. Thus we find some form of sharpness at the level of quantitative well-posedness between the spaces $\widehat{L}^p\cap L^2$ and $\widehat{L}^p$ for (\#). Similar sharpness at the level of well-posedness (in the sense of Hadamard) remains open, i.e.  whether one has a failure of continuity of the solution map from $(\widehat{L}^p)^N$ to $C([0,T),(\widehat{L}^p)^N)$.
\end{enumerate}

\noindent 
\textbf{Acknowledgement}: Both authors are  thankful to Prof. K. Sandeep   for the encouragement  during this project and his  thoughtful suggestions on the preliminary draft of this paper.  D.G. B is  thankful to DST-INSPIRE (DST/INSPIRE/04/2016/001507) for the financial  support.  Both authors are  thankful to TIFR CAM for the  excellent  research facilities. Both authors are also thankful to Prof. R\'emi Carles for  his thoughtful  suggestions and for bringing to our notice the reference \cite{hayashi1998asymptotics}.  Both authors are thankful to Prof. H.  Hajaiej for the encouragement and thoughtful  suggestions.  We are grateful to an anonymous referee for pointing out the error in the previous version of this paper. {This has helped us to improve the presentation of Section \ref{pill1}.}

\bibliographystyle{siam}
\bibliography{dgbsh}
\end{document}